\documentclass[12pt,reqno]{amsart}
\usepackage{graphicx}
\usepackage{cases}
\usepackage{amscd}
\usepackage{amsfonts}
\usepackage{amsmath}

\usepackage{subfigure}
\usepackage{graphicx}
\usepackage{listings}
\usepackage{subfigure}
\usepackage{amsthm}
\usepackage{cite}

\theoremstyle{remark}
\newtheorem{example}{\textbf{Example}}[section]
\numberwithin{equation}{section}

\usepackage{color}
\usepackage{datetime}

\allowdisplaybreaks

\usepackage[
pdfauthor={derajan},
pdftitle={How to do this},
pdfstartview=XYZ,
bookmarks=true,
colorlinks=true,
linkcolor=blue,
urlcolor=blue,
citecolor=blue,
pdftex,
bookmarks=true,
linktocpage=true, 
hyperindex=true
]{hyperref}
\usepackage{hyperref}
\allowdisplaybreaks

\usepackage{cleveref}

\makeatletter
\newcommand\figcaption{\def\@captype{figure}\caption}
\newcommand\tabcaption{\def\@captype{table}\caption}
\makeatother

\usepackage{listings}
\usepackage{geometry}
\usepackage{algorithm}
\usepackage{algorithmic}

\makeatletter
\newenvironment{NNalgorithm}[1][htb]{%
    \renewcommand{\ALG@name}{}
   \begin{algorithm}[#1]%
  }{\end{algorithm}}
\makeatother

\usepackage{multirow}

\def\bq{\begin{equation}}
\def\eq{\end{equation}}
\def\br{\begin{eqnarray}}
\def\er{\end{eqnarray}}
\def\brr{\bq\begin{array}{rlll}}
\def\err{\end{array}\eq}
\def\R{\mathbb{R}}
\def\text#1{\hbox{#1}}

\newtheorem{thm}{Theorem}[section]
\newtheorem{lem}{Lemma}[section]
\newtheorem{coro}{Corollary}[section]
\newtheorem{rem}{Remark}[section]

\newcommand{\bsub}{\begin{subequations}}
\newcommand{\esub}{\end{subequations}$\!$}

\title[]{Positive and free energy satisfying schemes for diffusion with interaction potentials}

\date{\today}
\author{ Hailiang Liu and Wumaier Maimaitiyiming}
\address{Iowa State University, Mathematics Department, Ames, IA 50011}
\email{hliu@iastate.edu}\email{wumaierm@iastate.edu} 
\date{}

\begin{document}

\begin{abstract}
In this paper, we design and analyze second order positive and free energy satisfying schemes for solving diffusion equations with interaction potentials. The semi-discrete scheme is shown to conserve mass, preserve solution positivity, and satisfy a discrete free energy dissipation law for nonuniform meshes.  These properties for the fully-discrete scheme (first order in time) remain preserved without a strict restriction on time steps.
For the fully second order (in both time and space) scheme, we use a local scaling limiter to restore solution positivity when necessary. It is proved that such limiter does not destroy the second order accuracy.  In addition, these schemes are easy to implement, and efficient in simulations over long time. Both one and two dimensional numerical examples are presented to demonstrate the performance  of these schemes.
\end{abstract}

\maketitle

\section{Introduction}

This paper is concerned with efficient numerical approximations to the following problem,
\begin{equation}\label{mainmodel}
\left \{
\begin{array}{lll}
&\partial_t \rho =\nabla \cdot (\nabla\rho+\rho \nabla (V(\mathbf{x})+W*\rho )),  \quad & \mathbf{x}\in \Omega\subset\R^d,  \quad  t>0, \\
& \rho(\mathbf{x},0)=\rho_0(\mathbf{x}), & \mathbf{x}\in \Omega\subset\R^d,
\end{array}
\right.
\end{equation}
subject to zero flux boundary conditions. Here $\Omega$ is a bounded domain in $\mathbb{R}^d$, $\rho=\rho(\mathbf{x},t)$ is the unknown density, $V(\mathbf{x})$ is a confinement potential, and $W(\mathbf{x})$ is an interaction potential, which is assumed to be symmetric.

Such problems appear in many applications. If $W$ vanishes, this model includes heat equation ($V(\mathbf{x})=0$) and the Fokker--Planck equation ($V(\mathbf{x})\neq 0 $, see e.g. \cite{Ri89}). With  interaction potentials, the equation can model nematic phase transition of rigid rod-like polymers \cite{Doi1}, chemotaxis \cite{Pe07}, and aggregation in biology (see \cite{AggW2,AggW1,Topez} and references therein). For chemotaxis, a wide literature exists in relation to the  Patlak-Keller-Segel system \cite{ KS70,Patlak}, and for rod-like polymers,   
the Doi-Onsager equation \cite{ Cons, Doi1, LiuZhang, Ons1949}  is a well studied model. 

Main properties of the solution to (\ref{mainmodel}) are non-negativity, mass conservation and free energy dissipation, i.e.,
\begin{equation}\label{positivity}
\rho_0(\mathbf{x})\geq 0  \Longrightarrow \rho(\mathbf{x},t)\geq 0, \quad  t>0,
\end{equation}

\begin{equation}\label{mass}
\int_{\Omega}\rho(\mathbf{x},t)d\mathbf{x}=\int_{\Omega}\rho_0(\mathbf{x})d\mathbf{x}, \quad  t>0,
\end{equation}

\begin{equation}\label{energydiss1}
\frac{dE(\rho)}{dt}=-\int _{\Omega}\rho|\nabla(\log(\rho)+V(x)+W*\rho )|^2d \mathbf{x}=-I(\rho)\leq 0,
\end{equation}
where the free energy associated to (\ref{mainmodel}) is given by
\begin{equation}\label{energy1}
E(\rho)=\int_{\Omega} \rho\log(\rho) d\mathbf{x}+\int_{\Omega}V(\mathbf{x})\rho d\mathbf{x}+\frac{1}{2}\int_{\Omega}\int_{\Omega}W(\mathbf{x}-\mathbf{y})\rho(\mathbf{y})\rho(\mathbf{x})d\mathbf{y}d\mathbf{x}.
\end{equation}
This energy functional is a sum of internal energy, potential energy, and the interaction energy. The functional $I$ is referred to as the entropy dissipation.
The nice mathematical features (\ref{positivity})-(\ref{energydiss1}) are crucial for the analytical study of (\ref{mainmodel}), while free-energy dissipation inequality (\ref{energydiss1}) is particularly important to understand the large time dynamics of solutions of (\ref{mainmodel})( see e.g., \cite{Energydiss1,Energydiss2, Mc97}). 
There have been many studies  about the connection between the free energy, the Fokker-Planck equation, and optimal transportation
in a continuous state space (see e.g., \cite{Br91, GM95, JO98, Ot01, Vi03}).

One way of obtaining a structure-preserving numerical scheme is the minimizing movement approximation (see \cite{Savar05} and the references therein), also named Jordan-Kinderlehrer-Otto (JKO) scheme (Jordan et al. \cite{JO98}), which is given by
$$
\rho^{n+1}=\text{argmin} \left\{\frac{1}{2\tau}W^2(\rho^n,\rho)+E(\rho)\right\}
$$
Here, at each time step, the distance of the solution update acts as a regularization to the free energy. Yet such problems involving the Wasserstein distance $W(\rho^n,\rho)$ are computationally demanding,  see, e.g., \cite{BCL16, CP18, LLW19, CCWW19} for some recent advances. 

The second way of obtaining a structure-preserving numerical scheme is by a direct discretization of (\ref{mainmodel}) so that these solution properties are preserved at the discrete level. This way has gained increasing attention in recent years, some closely related works include \cite{FiniteV, LY12, LY14, LY15, LW14, LW17, DG}.  In \cite{LY12},  second order implicit numerical schemes designed for linear (yet singular) Fokker-Planck equations satisfy all three solution properties without any time step restriction. In \cite{LW14}, the authors extended the idea in \cite{LY12} to a system of Poisson-Nernst-Planck equations using the explicit time discretization. 
For a more general class of  nonlinear nonlocal equations,
\begin{equation}\label{e2}
\partial_t \rho =\nabla \cdot \left(\rho \nabla (H'(\rho)+ V(\mathbf{x})+W*\rho )\right), 
\end{equation}
where $H$ is a smooth convex function, a second order finite-volume method was constructed in \cite{FiniteV}, where  positivity is enforced by using piecewise linear polynomials interpolating interface values.  Structure preserving schemes based on the Chang-Cooper scheme \cite{Cooper70} have been  constructed in \cite{PZ18} to numerically solve nonlinear Fokker-Planck equations. Note that in \cite{LW14, FiniteV, PZ18} different time step  restrictions are imposed in order to preserve the desired solution properties.  

The construction of higher order schemes using the discontinuous Galerkin (DG) framework has recently been carried out for Fokker-Planck-type equations. We refer to \cite{LY15} for entropy satisfying DG  schemes of arbitrary high order, and  to \cite{LY14} for a DG scheme of third order to satisfy the discrete maximum principle for linear Fokker-Planck equations.  In \cite{LW17}, the authors  designed free energy satisfying DG schemes of any high order for Poisson-Nernst-Planck equations, but positive cell averages are shown to propagate in time only  for special cases. While in \cite{DG}, a high order nodal DG method for  (\ref{e2}) was constructed using $k+1$ Gauss--Lobatto quadrature points for degree $k$ polynomials in order to preserve both the entropy dissipation and the solution positivity; somehow degeneracy of accuracy in some cases was reported. Despite some well-known advantages of the DG method,  structural properties of the above fully discrete DG schemes are verified under some CFL conditions. It would be interesting to explore some explicit-implicit strategies for DG schemes. 


In this paper we extend the idea in \cite{LY12} to construct explicit-implicit schemes which are proven to preserve three main properties of (\ref{mainmodel}) without a strict restriction on time steps.  This therefore has improved upon the work \cite{LW14}.  Our main results include the scheme formulation, proofs of mass conservation, solution non-negativity,  and the discrete free-energy dissipation law for both semi-discrete and fully discrete methods. In particular, the fully-discrete scheme (first order in time) is shown to satisfy three desired properties without strict restriction on time steps, in both one and two dimensional cases with nonuniform meshes.  For the fully second order (in both time and space) scheme, we apply a local scaling limiter to restore solution positivity, such limiter was first introduced in \cite{LMJCAM}, in this paper we rigorously  prove that such limiter does not destroy the second order accuracy. 

More precisely, our scheme construction is based on a reformulation 
\begin{equation}\label{mainmodel2}
\partial_t \rho=\nabla\cdot \left(M\nabla  \left( \frac{\rho}{M}\right)\right),
\end{equation}
where $M=e^{-V(\mathbf{x})-W*\rho}$,  motivated by the fact that the equilibrium solutions of (\ref{mainmodel}) may be expressed as $\rho=C e^{-V(\mathbf{x})-W*\rho}$. For linear Fokker-Planck equations,  such reformulation with $M=e^{-V(\mathbf{x})}$ (so called non-logarithmic Landau form) has been used in \cite{LY12},  as well as in earlier works ( see e.g., \cite{LAN}). We note that for the general nonlinear nonlocal model (\ref{e2}), our scheme construction remains valid if we take $M=\rho e^{-H'(\rho)-V(\mathbf{x})-W*\rho}$ in the reformulation (\ref{mainmodel2}).

The advantage of formulation (\ref{mainmodel2}) can be seen from both spatial and temporal discretization. The symmetric spatial discretization of the one-dimensional version of (\ref{mainmodel2}) 
yields the semi-discrete scheme 
\begin{equation}\label{semi01}
h_j \frac{d}{dt}{\rho} _j= h^{-1}_{j+1/2} M_{j+1/2}\left( \frac{\rho_{j+1}}{M_{j+1}} -  \frac{\rho_j}{M_j} \right)- h^{-1}_{j-1/2} M_{j-1/2}\left( \frac{\rho_{j}}{M_{j}} -  \frac{\rho_{j-1}}{M_{j-1}} \right), 
 \end{equation}
 in which the evaluation of $M$ at cell interfaces $\{x_{j+1/2}\}$ and cell centers $\{x_j\}$ is easily available as defined in 
 (\ref{MGsemi}). 
 Here $\rho_j$ approximates the cell average of $\rho(x,t)$ on $j$-th computational cell $[x_{j-1/2}, x_{j+1/2}]$ of size $h_j$, and  $h_{j+1/2}=(h_j+h_{j+1})/{2}$.

For time discretization of (\ref{semi01}), 
we adopt an implicit-explicit approach to obtain 
\begin{equation}\label{fully01}
h_j \frac{\rho_j^{n+1} -\rho_j^n}{\tau}= h^{-1}_{j+1/2} M^n_{j+1/2}\left( \frac{\rho^{n+1}_{j+1}}{M_{j+1}^n} -  \frac{\rho^{n+1}_j}{M_j^n} \right) - h^{-1}_{j-1/2}M^n_{j-1/2}\left( \frac{\rho^{n+1}_{j}}{M_{j}^n} -  \frac{\rho^{n+1}_{j-1}}{M_{j-1}^n}\right),
\end{equation}
where $\rho_j^n$ approximates ${\rho}_j(t)$ at time $t=n\tau$, see (\ref{fully}). This scheme is easy to implement, and is shown to preserve all three desired properties without a strict time step restriction.  However, the scheme (\ref{fully01}) is only first order in time. We further 
propose a fully second order scheme: 
\begin{equation}\label{CN}
\begin{aligned}
& h_j\frac{\rho^*_j-\rho^n_j}{\tau/2}= h^{-1}_{j+1/2} M^*_{j+1/2}\left( \frac{\rho^{*}_{j+1}}{M_{j+1}^*} -  \frac{\rho^{*}_j}{M_j^*} \right) - h^{-1}_{j-1/2}M^*_{j-1/2}\left( \frac{\rho^{*}_{j}}{M_{j}^*} -  \frac{\rho^{*}_{j-1}}{M_{j-1}^*}\right),\\
& \rho^{n+1}_j=2\rho^*_j-\rho^n_j,
\end{aligned}
\end{equation}
based on the predictor-corrector methodology, where $M^*_{j}$  and  $M^*_{j+1/2}$ are given in (\ref{PC1}).
This scheme is second order in both time and space,  and it preserves solution positivity for small time steps. For large time steps, we use a local scaling limiter to restore the solution positivity. 
 
Although we derive the schemes for the model equation (\ref{mainmodel}), the methods can be easily applied to a larger class of problems where the solution depends on additional parameters and the PDE is of drift-diffusion type; see \cite{LM19}. 
  
The rest of the paper is organized as follows. In section 2, we present a semi-discrete scheme for one dimensional problems. Theoretical analysis of three properties is provided. In section 3, we present fully discrete implicit-explicit schemes for one dimensional case and prove the desired properties. Section 4 is devoted to numerical
schemes for two dimensional problems. In section 5, we extend the scheme to a fully  second order (in both time and space) scheme, a mass conserving local limiter is also introduced to restore solution positivity. 
Numerical examples for one and two dimensional problems are presented in section 6.  Finally,  concluding remarks are given in section 7.


\section{Numerical Method: one dimensional case}
We begin with
\begin{equation}\label{mainmodel11}
\left \{
\begin{array}{ll}
\partial_t \rho =\partial_x (\partial_x \rho +\rho \partial_x (V(x)+W*\rho )),  \quad & x\in \Omega,  \quad  t>0, \\
  \rho(x,0)=\rho_0(x),  \quad & x\in \Omega, \\
  \partial_x \rho+ \rho \partial_x (V(x)+W*\rho )=0,  \quad & x\in \partial \Omega,  \quad  t>0. \\
\end{array}
\right.
\end{equation}
and  reformulate (\ref{mainmodel11}) as
\begin{equation}\label{mainmodel12}
\left \{
\begin{array}{ll}
\partial_t \rho =\partial_x (M \partial_x (\rho/ M)),  \quad & x\in \Omega,  \quad  t>0, \\
    \rho(x,0)=\rho_0(x),  \quad & x\in \Omega, \\
M \partial_x (\rho/M)=0,  \quad & x\in \partial \Omega,  \quad  t>0,  \\
\end{array}
\right.
\end{equation}
where $M=e^{-V(x)-W*\rho}$. 
We propose a finite volume scheme for (\ref{mainmodel12}) over the interval $\Omega=[a, \ b]$.
For a given positive integer $N$, we partition domain $\Omega$ into computational cells $I_j=[x_{j-\frac{1}{2}}, \ \ x_{j+\frac{1}{2}}]$ with mesh size $h_j=|I_j|$ and cell center at $x_j=x_{j-\frac{1}{2}}+\frac{1}{2}h_j$, $j\in \{1,2,\cdots, N\},$ we set  $h_{j+1/2}=(h_j+h_{j+1})/{2}$.

\subsection{Semi-discrete scheme}
We integrate on each computational cell $I_j$ to obtain
$$
\frac{d}{dt} \int_{I_j}\rho(x,t) dx=M\partial_x (\rho/M)|_{ x_{j+1/2}}- M\partial_x (\rho/M)|_{ x_{j-1/2}}.
$$
Let $\rho(t)=\{\rho_1,\ \cdots,\rho_N\}$ be the numerical solution approximating all cell averages and $C_{j+1/2}$ be an approximation to $M\partial_x (\rho/M)|_{ x_{j+1/2}}$, then one has the following semi-discrete scheme,
\begin{equation}\label{semi}
\begin{aligned}
\frac{d}{dt}\rho _j&=\frac{C_{j+1/2}  - C_{j-1/2}}{ h_j },  \ \ j=1,2,\cdots, N,
 \end{aligned}
\end{equation}
 we define 
\begin{align*}
& C_{j+1/2}= \frac{M_{j+\frac{1}{2}}}{h_{j+1/2}}\left( \frac{\rho_{j+1}}{M_{j+1}}-\frac{\rho_j}{M_j}\right) \ \ \ \text{for } \ j=1,2,\cdots, N-1,\\ \label{ZFsemi}
&C_{1/2}=0,\ \ C_{N+1/2}=0.
\end{align*}
 Here $M_{j+1/2}=Q_1(x_{j+1/2}, \rho)$ and $M_j=Q_1(x_j,\rho)$ with
\begin{equation}\label{MGsemi}
Q_1(x,v)=e^{-V(x)-\sum_{i=1}^Nh_iW(x_i-x)v_{i}},\quad \text{ for } x\in \R, \ v\in \R^{N}.
\end{equation}
Note that the zero flux boundary conditions have been weakly enforced. 

\subsection{Scheme properties} We investigate three desired properties for this semi-discrete scheme. For the energy dissipation property, we 
define a semi-discrete version of the free energy (\ref{energy1}) as
\begin{equation}\label{semienergy}
E_{h}(t)=\sum_{j=1}^Nh_j\left(\rho_j\log(\rho_j)+V_j\rho_j+\frac{1}{2}g_j \rho_j \right),
\end{equation}
where $g_j=\sum_{i=1}^Nh_iW(x_i-x_j)\rho_i$ is a second order approximation of the convolution $(W*\rho)(x_j).$

The following theorem states that the semi-discrete scheme (\ref{semi}) is conservative, positive, and energy dissipating.
\begin{thm} \label{th2.1}The semi-discrete scheme (\ref{semi}) satisfies the following properties:\\
(1) Conservation of mass: for any $ t>0$ we have
\begin{equation}\label{mass21}
\sum_{j=1}^{N} h_j\rho_j(t)=\sum_{j=1}^{N} h_j \rho_j(0).
\end{equation}
(2) Positivity preserving:  if $\rho_j(0)\geq 0$ for all $j\in \{1,\cdots,N\}$, then $\rho_j(t)\geq 0$ for any $t>0$. \\
(3) Entropy dissipation: $\frac{dE_{h}(t)}{dt}\leq -I_{h}$, where 
\begin{equation}\label{dissipation}
I_{h}=\sum_{j=1}^{N-1}C_{j+1/2}\left (\log{(\frac{\rho_{j+1}}{M_{j+1}})}-\log{(\frac{\rho_j}{M_j})} \right)\geq 0.
\end{equation}
\end{thm}

\begin{proof}
(1) Summing all equations in (\ref{semi}), we have
$$\frac{d}{dt}\sum_{j=1}^Nh_j \rho_j(t)=\sum_{j=1}^N\frac{d}{dt} h_j \rho_j(t)=0,$$
therefore (\ref{mass21}) holds true for any $ t>0.$

(2) Let $\vec{F}(\vec{\rho})$ be the vector field defined by the right hand side of (\ref{semi}), then
\begin{equation}\label{system}
\frac{d}{dt}\vec{\rho}=\vec{F}(\vec{\rho}).
\end{equation}
Note that the hyperplane
$\Sigma=\{\vec{\rho}: \sum_{j=1}^{N}h_j\rho_j=\sum_{j=1}^{N}h_j\rho_j(0) \}$
is an invariant region of (\ref{system}). We define a closed set $\Sigma_1$ on this hyperplane by
$$\Sigma_1=\bigg \{\vec{\rho}: \rho_j\geq 0, j=1,2,\cdots,N, \text{and} \ \sum_{j=1}^{N}h_j\rho_j=\sum_{j=1}^{N}h_j\rho_j(0) \bigg\}.$$
It suffices to show that $\Sigma_1$ is invariant under system (\ref{system}). This is the case if the vector field $\vec{F}(\vec{\rho})$ strictly points to interior of $\Sigma_1$ on its boundary $\partial \Sigma_1$: i.e.,
$$\vec{F}(\vec{\rho})\cdot \vec{v}<0,$$
where $\vec{v}$ is outward normal vector on any part of $\partial \Sigma_1$.

A direct calculation using (\ref{semi}) gives
\begin{equation}\label{vecfield}
\begin{aligned}
\vec{F}(\vec{\rho})\cdot \vec{v}&=\sum_{j=1}^{N-1}\frac{v_j}{h_{j}}C_{j+1/2}-\sum_{j=2}^{N}\frac{v_j}{h_{j}}C_{j-1/2}\\
&=- \sum_{j=1}^{N-1}(\frac{v_{j+1}}{h_{j+1}}-\frac{v_j}{h_j})C_{j+1/2}.
\end{aligned}
\end{equation}
For each $\vec{\mu}\in \partial\Sigma_1$, we define the set
$S=\{j: 1\leq j\leq N \text{ and } \  {\mu}_j=0\},$
then the outward normal vector at $\vec{\mu}$ has the form
$$\vec{v}=(v_1,v_2,\cdots,v_N)^{T} \ \ \text{with }
v_i=\left \{
\begin{array}{ll}
-\alpha_i,  \quad & i\in S, \\
  0,  \quad & i \notin S,
\end{array}
\right.
$$
and $\alpha_i>0$ if $i\in S.$

Note that if $j,\ j+1\in S,$ then $\rho_{j}=\rho_{j+1}=0$ implies $C_{j+1/2}=0$; if $j,\ j+1\notin S$, then $v_{j+1}=v_j=0$. Therefore nonzero terms in (\ref{vecfield}) are those with $j\in S, \ j+1\notin S$ or $j\notin S, \ j+1\in S.$ Hence
\begin{align*}
\vec{F}(\vec{\rho})\cdot \vec{v}&=-\sum_{j\in S,  j+1\notin S}\frac{\alpha_j}{h_j}\frac{M_{j+\frac{1}{2}}}{h_{j+1/2}}\frac{\rho_{j+1}}{M_{j+1}}- \sum_{j\notin S,  j+1\in S}\frac{\alpha_{j+1}}{h_{j+1}}\frac{M_{j+\frac{1}{2}}}{h_{j+1/2}}\frac{\rho_j}{M_j}<0.
\end{align*}
Therefore $\Sigma_1$ is an invariant region of (\ref{semi}), this completes the proof of (2).

(3) From the fact that $W(x)=W(-x)$, it follows 
\begin{equation}\label{convolutionderivative}
\frac{d}{dt}\sum_{j=1}^N \frac{h_j}{2}g_j \rho_j =\sum_{j=1}^N h_jg_j \frac{d\rho_j}{dt}. 
\end{equation}
Differentiating the discrete free energy (\ref{semienergy}) with respect to time and using (\ref{convolutionderivative}) we obtain
\begin{align*}
\frac{dE_{h}(t)}{dt}&=\sum_{j=1}^N(\log(\rho_j)+1+V_j+g_j )h_j\frac{d\rho_j}{dt} \\
&= \sum_{j=1}^N(
\log{(\frac{\rho_j}{M_j})}  +1)(C_{j+1/2}- C_{j-1/2})\\
&=-\sum_{j=1}^{N-1}C_{j+1/2}\left(\log{(\frac{\rho_{j+1}}{M_{j+1}})}-\log{(\frac{\rho_j}{M_j})} \right)\\
&=-I_{h}\leq 0.
\end{align*}
Note that
\begin{align*}
I_h=& \sum_{j=1}^{N-1}C_{j+1/2}\left (\log{(\frac{\rho_{j+1}}{M_{j+1}})}-\log{(\frac{\rho_j}{M_j})} \right)\\
       =& \sum_{j=1}^{N-1}\frac{1}{h_{j+1/2}}M_{j+1/2}\left (\frac{\rho_{j+1}}{M_{j+1}}-\frac{\rho_j}{M_j}\right) \left(\log{(\frac{\rho_{j+1}}{M_{j+1}})}-\log{(\frac{\rho_j}{M_j})} \right)
\end{align*}
and
 $(x-y)(\log{x}-\log{y})\geq 0 \;  \text{ for }  x,y\in \R^+,$  so we have $I_{h}\geq 0.$
\end{proof}


\section{Fully discrete scheme} For time discretization of (\ref{semi}), we use an implicit-explicit time discretization in order to construct an easy to implement yet stable numerical scheme without time step restriction.

\subsection{Scheme formulation and algorithm}
Let $\tau$ be time step and $\rho_j^n$ be the numerical solution at $t_n=n\tau$ to  approximate $\rho_j(t_n).$  From given $\rho_j^n$, $j=1,2,\cdots, N$, we update to get $\rho_j^{n+1}$ by

\begin{equation}\label{fully}
\begin{aligned}
\frac{\rho^{n+1}-\rho^n}{\tau}&=\frac{C^{n,*}_{j+1/2}  - C^{n,*}_{j-1/2}}{ h_j },  \ \ j=1,2,\cdots, N.
\end{aligned}
\end{equation}
with 
\begin{align*}
& C^{n,*}_{j+1/2}= \frac{M^n_{j+\frac{1}{2}}}{h_{j+1/2}}\left ( \frac{\rho^{n+1}_{j+1}}{M^n_{j+1}}-\frac{\rho^{n+1}_j}{M^n_j} \right) \ \ \ \text{for } \ j=1,2,\cdots, N-1,\\ \label{ZFsemi}
&C^{n,*}_{1/2}=  C^{n,*}_{N+1/2}=0,
\end{align*}
where $M^n_{j+1/2}=Q_{1}(x_{j+1/2}, \rho^n)$ and $M_j^{n}=Q_1(x_j, \rho^n)$.
The initial data is chosen by 
\begin{equation}\label{ini1}
\rho_{j}^0=\frac{1}{h_j}\int_{I_j}\rho_0(x)dx, \ \ j=1,2,\cdots, N.
\end{equation}

\subsection{Scheme properties} 
Define a fully discrete version $E^n_{h}$ of the free energy (\ref{energy1})  as
\begin{equation}\label{energyfully}
E_{h}^n=\sum_{j=1}^Nh_j \left(\rho^{n}_j\log (\rho^{n}_j)+V_j\rho^{n}_j+\frac{1}{2}g^{n}_j \rho^{n}_j \right),
\end{equation}
where $g_j^n=\sum_{j=1}^Nh_iW(x_i-x_j)\rho_i^n.$  

The following theorem states that the three desired properties are preserved by the scheme (\ref{fully}) without strict time step restriction.


\begin{thm} \label{thm3.1} The fully discrete scheme (\ref{fully}) has the following properties:

(1) Conservation of mass:
\begin{equation}\label{mass2}
\sum_{j=1}^Nh_j\rho_j^{n}= \int_{\Omega} \rho_0(x)dx \ \ \text{ for } n\geq 1.
\end{equation}

(2) Positivity preserving: if $\rho_j^n\geq0$ for all $j=1,\cdots, N,$ then
$$\rho_j^{n+1}\geq 0, \quad j=1,\cdots,N.$$

(3) Entropy dissipation: there exists $\tau^*>0$ such that if $\tau\in(0,\  \tau^*)$, then 
\begin{equation*}\label{FE}
E_{h}^{n+1}-E_{h}^n\leq -\frac{\tau}{2} I^{n}_{h},
\end{equation*}
where
\begin{equation*}\label{dissipationfully}
I^n_{h}=\sum_{j=1}^{N-1}C^{n,*}_{j+1/2}\left (\log{(\frac{\rho^{n+1}_{j+1}}{M^n_{j+1}})}-\log{(\frac{\rho^{n+1}_j}{M^n_j})} \right)\geq 0.
\end{equation*}
\end{thm}
\begin{proof} 
Set $G^{n,*}_{j}=\rho^{n+1}_j/M^n_{j}$ and $\lambda_{j+1/2}=\tau/h_{j+1/2}$, 
so the fully discrete scheme (\ref{fully}) can be rewritten into the following linear system:
\begin{equation}\label{fullyequal}
\begin{aligned}
 h_1\rho_1^n=&(h_1M_1^n+\lambda_{1+1/2} M^n_{1+1/2} )G^{n,*}_1-\lambda_{1+1/2} M^n_{1+1/2}G^{n,*}_{2} ,\\
 h_j\rho_j^n=&-\lambda_{j-1/2} M^n_{j-1/2}G^{n,*}_{j-1}+(h_jM_j^n+\lambda_{j-1/2} M^n_{j-1/2}+\lambda_{j-1/2} M^n_{j+1/2} )G^{n,*}_j\\
 &-\lambda_{j+1/2} M^n_{j+1/2}G^{n,*}_{j+1}  \ \ \ \ \ \ \ \  j=2,3,\cdots , N-1,\\
 h_N\rho_N^n =&-\lambda_{N-1/2} M^n_{N-1/2}G^{n,*}_{N-1}+(h_NM_N^n+\lambda_{N-1/2} M^n_{N-1/2})G^{n,*}_{N}.
\end{aligned}
\end{equation}
Note that the coefficient matrix of linear system (\ref{fullyequal}) is strictly diagonally dominant, therefore (\ref{fullyequal}) has a unique solution for whatever $\tau$ a priori chosen so dose (\ref{fully}) because $\rho^{n+1}_j=G^{n,*}_jM^n_j$. 

(1) (\ref{mass2}) follows from adding all equations in system (\ref{fullyequal}) and using (\ref{ini1}).

(2) Since $\rho_j^{n+1}=M_j^nG_j^{n,*}$ and $M_j^n>0$,  it suffices to prove that 
$$
G_i^{n,*}=\min_{1\leq j\leq N} {\{G_j^{n,*}\}}\geq0.
$$ 
Assume $1<i<N,$ from $i$-th equation of (\ref{fullyequal}) we have
\begin{align*}
h_i\rho_i^n &=-\lambda_{i-1/2} M^n_{i-1/2}G^{n,*}_{i-1}+(h_iM_i^n+\lambda_{i-1/2} M^n_{i-1/2}+\lambda_{i+1/2} M^n_{i+1/2} )G^{n,*}_i-\lambda_{i+1/2} M^n_{i+1/2}G^{n,*}_{i+1} \\
&\leq -\lambda_{i-1/2} M^n_{i-1/2}G^{n,*}_{i}+(h_iM_i^n+\lambda_{i-1/2} M^n_{i-1/2}+\lambda_{i+1/2} M^n_{i+1/2} )G^{n,*}_i-\lambda_{i+1/2} M^n_{i+1/2}G^{n,*}_{i} \\
&=h_iM_i^nG^{n,*}_i.
\end{align*}
Thus $G_i^{n,*}\geq \frac{\rho_i^n}{M_i^n}\geq 0.$ A similar argument applies if $i=1$ or $i=N.$

(3) A direct calculation using (\ref{energyfully}) gives 
\begin{align*}
E_{h}^{n+1}-E_{h}^n =&\sum_{j=1}^Nh_j\left(\rho^{n+1}_j\log(\rho^{n+1}_j)-\rho^{n}_j\log(\rho^{n}_j\right)+V_j\rho^{n+1}_j-V_j\rho^{n}_j+\frac{1}{2}g^{n+1}_j\rho^{n+1}_j -\frac{1}{2}g^{n}_j
\rho^{n}_j )\\
=&\sum_{j=1}^Nh_j((\rho^{n+1}_j-\rho^{n}_j)\log(\rho_j^{n+1})+(\rho^{n+1}_j-\rho^{n}_j)V_j+(\rho^{n+1}_j-\rho^{n}_j)g_j^n\\
&+\frac{1}{2}g_j^n\rho_j^n-g_j^n\rho_j^{n+1}+\frac{1}{2}g_j^{n+1}\rho_j^{n+1}+\rho_j^n\log(\frac{\rho_j^{n+1}}{\rho_j^n}))\\
\leq&\sum_{j=1}^Nh_j((\rho^{n+1}_j-\rho^{n}_j)\log(G_j^{n,*})+\frac{1}{2}g_j^n\rho_j^n-g_j^n\rho_j^{n+1}+\frac{1}{2}g_j^{n+1}\rho_j^{n+1}),
\end{align*}
here we have used $\rho_j^n\log(\frac{\rho_j^{n+1}}{\rho_j^n})\leq \rho_j^n(\frac{\rho_j^{n+1}}{\rho_j^n}-1)$ and mass conservation $\sum_{j=1}^Nh_j(\rho^{n+1}_j-\rho^{n}_j)=0.$ We proceed with
\begin{equation}\label{F1}
\begin{aligned}
\tau \sum_{j=1}^N (\frac{h_j\rho^{n+1}_j-h_j\rho^{n}_j}{\tau})\log(G_j^{n,*}) 
&=\tau\sum_{j=1}^N(\log(G_j^{n,*})(h^{-1}_{j+1/2}M^{n}_{j+1/2}(G^{n,*}_{j+1}-G^{n,*}_j)\\
&\qquad - h^{-1}_{j-1/2}M^{n}_{j-1/2}(G^{n,*}_j-G^{n,*}_{j-1})))\\
&= -\tau\sum_{j=1}^{N-1}h^{-1}_{j+1/2}M^{n}_{j+1/2}(G_{j+1}^{n,*}-G_j^{n,*})(\log{G_{j+1}^{n,*}-\log{G_j^{n,*}}})\\
&=-\tau I^{n}_{h} \leq 0.
\end{aligned}
\end{equation}
Here the sign of $I_h^n$ is implied by the monotonicity of the logarithmic function. 

It remains to find a sufficient condition on time step $\tau$ so that 
\begin{equation}\label{final}
\sum_{j=1}^Nh_j(\frac{1}{2}g_j^n\rho_j^n-g_j^n\rho_j^{n+1}+\frac{1}{2}g_j^{n+1}\rho_j^{n+1})\leq -\frac{\tau}{2}\sum_{j=1}^N(\frac{h_j\rho^{n+1}_j-h_j\rho^{n}_j}{\tau})\log(G_j^{n,*}).
\end{equation}
From $\sum_{j=1}^N h_jg_j^n\rho_j^{n+1} =\sum_{j=1}^N h_jg_j^{n+1}\rho_j^{n}$ it follows that 
\begin{align*}
\sum_{j=1}^N h_j(\frac{1}{2}g_j^n\rho_j^n-g_j^n\rho_j^{n+1}+\frac{1}{2}g_j^{n+1}\rho_j^{n+1})
&=\frac{1}{2}\sum_{j=1}^Nh_j(g_j^{n+1}-g_j^n)(\rho_j^{n+1}-\rho_j^{n})\\
&=\frac{1}{2}\sum_{j=1}^Nh_j\sum_{i=1}^N h_iW(x_i-x_j)(\rho_i^{n+1}-\rho_i^{n})(\rho_j^{n+1}-\rho_j^{n})\\
&\leq \frac{||W||_\infty }{2}\sum_{j=1}^Nh_j\sum_{i=1}^Nh_i|\rho_i^{n+1}-\rho_i^{n}||\rho_j^{n+1}-\rho_j^{n}|\\
&\leq \frac{||W||_\infty(b-a) \tau^2}{2}\sum_{j=1}^Nh_j\left(\frac{\rho_j^{n+1}-\rho_j^{n}}{\tau}\right)^2,
\end{align*}
where we have used the Cauchy-Schwarz inequality and $b-a=\sum_{j=1}^Nh_j.$ Let $\vec{\xi}, \vec{\eta}\in \R^N$ be vectors defined as $\vec{\xi}_j=\frac{\sqrt{h_j}(\rho_j^{n+1}-\rho_j^n)}{\tau}$, $\vec{\eta}_j=\sqrt{h_j}\log{G_j^{n,*}}$, then (\ref{final}) is satisfied if
$$
\frac{||W||_\infty(b-a) \tau^2}{2}|\vec{\xi}|^2+\frac{\tau }{2}\vec{\xi}\cdot \vec{\eta}\leq 0.
$$
We claim that 
\begin{equation}\label{1/2}
 \vec{\xi}\cdot \vec{\eta}=0 \quad \text{ if and only if }\quad \vec{\xi}=0. 
 \end{equation}
 Therefore 
$$
0< c_0 \leq \frac{-\vec{\xi}\cdot \vec{\eta}}{|\vec{\xi}|^2}\leq \frac{|\eta|}{|\xi|} \quad \text{for} \; \xi \not=0,
$$
where $c_0$ may depend on numerical solutions at $t_n$ and $t_{n+1}$. 
We thus obtain (\ref{final}) by taking
$$\tau\leq \tau^*=\frac{c_0}{||W||_{\infty}(b-a)}.
$$
Finally, we verify claim (\ref{1/2}).  If $\vec{\xi}\cdot \vec{\eta}=0,$  then from (\ref{F1}) we have 
$$
0=\vec{\xi}\cdot \vec{\eta}= -\tau\sum_{j=1}^{N-1}h^{-1}_{j+1/2}M^{n}_{j+1/2}(\log{G_{j+1}^{n,*}-\log{G_j^{n,*}}})(G_{j+1}^{n,*}-G_j^{n,*}) \leq 0,
$$
therefore we must have $G_j^{n,*}=constan$ for all $j\in \{1,2,\cdots, N\}$.  This when inserted into scheme (\ref{fully}) leads to 
$$
\rho_j^{n+1}=\rho_j^n  \text{ for all }  j\in \{1,2,\cdots,N\},
$$
thus $\vec{\xi}=0.$ 
\end{proof}

\begin{rem} One could take the Euler forward time discretization to obtain an explicit scheme: 
From $\rho_j^n$, $j=1,2,\cdots, N$, update to get $\rho_j^{n+1}$ by 
\begin{equation*}\label{fullyexp}
\begin{aligned}
\frac{\rho^{n+1}-\rho^n}{\tau}&=\frac{C^{n}_{j+1/2}  - C^{n}_{j-1/2}}{ h_j },  \ \ j=1,2,\cdots, N.
\end{aligned}
\end{equation*}
where
\begin{align*}
& C^{n}_{j+1/2}= \frac{M^n_{j+\frac{1}{2}}}{h_{j+1/2}}\left( \frac{\rho^{n}_{j+1}}{M^n_{j+1}}-\frac{\rho^{n}_j}{M^n_j} \right) \ \ \ \text{for } \ j=1,2,\cdots, N-1,\\ \label{ZFsemi}
&C^{n}_{1/2}=  C^{n}_{N+1/2}=0,
\end{align*}
with $M^n_{j+1/2}=Q_1(x_{j+1/2},\rho^n)$ and $M_j^{n}=Q_1(x_j, \rho^n)$. 
One can show that the positivity preserving property is still met  yet under a CFL condition like $\tau \leq \gamma h^2$. 
\end{rem}


\section{Numerical Method: two dimensional Case} In this section, we extend our method to multi-dimensional problems. For simplicity, 
we only present schemes for the two dimensional initial value problem,

\begin{equation}\label{mainmodel21}
\left \{
\begin{array}{lll}
&\partial_t \rho =\nabla \cdot (\nabla \rho+\rho \nabla (V(x,y)+W*\rho )),  \quad &  (x,y)\in \Omega\subset\R^2,  \quad  t>0, \\
& \rho(x,y,0)=\rho_0(x,y) ,  \quad & (x,y)\in \Omega,
\end{array}
\right.
\end{equation}
on a rectangular domain $\Omega=[a\ , \ b]\times [c\ , \ d]$ subject to zero flux boundary conditions.

For given positive integers $N_x,N_y$,  we partition $\Omega$ by a Cartesian mesh with computational cells 
$$
I_{i,j}=[x_{i-\frac{1}{2}}, \ \ x_{i+\frac{1}{2}}]\times [y_{j-\frac{1}{2}}, \ \ y_{j+\frac{1}{2}}],
$$ 
where $i\in \{1,2,\cdots, N_x\}, j\in \{1,2,\cdots, N_y\}.$  The mesh size is $|I_{i,j}|=h^x_ih^y_j$ with the cell center at $(x_i,\ y_j)=(x_{i-1/2}+\frac{1}{2}h^x_i, y_{j-1/2}+\frac{1}{2}h^y_j)$, we set $h^x_{i+1/2}=(h^x_i+h^x_{i+1})/2,$  $h^y_{j+1/2}=(h^y_j+h^y_{j+1})/2$.
\subsection{Semi-discrete scheme} 
Let $\rho(t)=\{\rho_{i,j}\}$ be the numerical solution, then dimension by dimension spatial discretization of 
$$
\partial_t \rho=\nabla\cdot\left(M\nabla (\frac{\rho}{M})\right), \quad \text{ with } M=e^{-V(x,y)-W*\rho},
$$ 
yields  the following semi-discrete scheme 
\begin{equation}\label{Nosemi2}
\frac{d}{dt}\rho _{i,j}=\frac{C_{i+1/2,j}-C_{i-1/2,j}}{h^x_i}+\frac{C_{i,j+1/2}-C_{i,j-1/2}}{h^y_j},
\end{equation}
where 
\begin{equation*}\label{C2}
\begin{aligned}
& C_{i+1/2,j}=\frac{M_{i+1/2,j}}{h^x_{i+1/2}}\bigg(\frac{\rho_{i+1,j}}{M_{i+1,j}}-    \frac{\rho_{i,j}}{M_{i,j}}   \bigg), \quad i=1,\cdots,N_x-1, j=1,\cdots,N_y,\\
& C_{i,j+1/2}=\frac{M_{i,j+1/2}}{h^y_{j+1/2}}\bigg(\frac{\rho_{i,j+1}}{M_{i,j+1}}-    \frac{\rho_{i,j}}{M_{i,j}}   \bigg), \quad i=1,\cdots,N_x, j=1,\cdots,N_y-1,\\
& C_{1/2,j}=C_{N_x+1/2,j}=C_{i,1/2}=C_{i,N_y+1/2}=0, \quad i=1,\cdots,N_x, j=1,\cdots,N_y,
\end{aligned}
\end{equation*}
with ${M}_{i+1/2,j}=Q_2(x_{i+1/2}, y_{j},\rho )$, ${M}_{i,j+1/2}=Q_2(x_{i}, y_{j+1/2},\rho )$, and $M_{i,j}=Q_2(x_{i}, y_{j},\rho )$. Where
\begin{equation}\label{QQ2}
Q_2(x, y, v)=e^{-V(x, y)-\sum_{k=1}^{N_{x}}\sum_{l=1}^{N_{y}}h^x_kh^y_lW(x_k-x, y_{l}-y ) v_{k,l}}, \text{ for } x, y\in \R, v\in \R^{N_x\times N_y}.
\end{equation}
Let 
\begin{equation*}\label{semienergy2}
E_{h}(t)= \sum_{i=1}^{N_{x}}\sum_{j=1}^{N_{y}}h^x_ih^y_j \left(\rho_{i,j}\log(\rho_{i,j})+V_{i,j}\rho_{i,j}+\frac{1}{2}g_{i,j}\rho_{i,j}\right),
\end{equation*}
be an approximation of the entropy functional (\ref{energy1}), with 
$$
g_{i,j}=\sum_{k=1}^{N_{x}}\sum_{l=1}^{N_{y}}h^x_kh^y_lW(x_k-x_i,y_{l}-y_j)\rho_{k,l}.
$$ 
The following theorem states that the semi-discrete scheme (\ref{Nosemi2}) is conservative, positive,  and energy dissipating.
\begin{thm} The semi-discrete scheme (\ref{Nosemi2}) satisfies the following properties:\\
(1) Conservation of mass: for any $ t>0,$
\begin{equation*}\label{mass41}
\sum_{i=1}^{N_{x}}\sum_{j=1}^{N_{y}}h^x_ih^y_j\rho_{i,j}(t)= \sum_{i=1}^{N_{x}}\sum_{j=1}^{N_{y}}h^x_ih^y_j\rho_{i,j}(0).
\end{equation*} 
(2) Positivity preserving: if $\rho_{i,j}(0)\geq 0$ for all $i\in\{1,\cdots,N_x\}\ ,j\in \{1,\cdots,N_y\}$, then $\rho_{i,j}(t)\geq 0$ for any $t>0.$\\
(3) Entropy dissipation: $\frac{dE_{h}(t)}{dt}\leq -I_{h}$, where
\begin{equation*}\label{dissipation2}
\begin{aligned}
I_{h}=&\sum_{j=1}^{N_{y}}\sum_{i=1}^{N_{x}-1}h^y_jC_{i+1/2,j}\left (\log({\frac{\rho_{i+1,j}}{M_{i+1,j}}})-\log{( \frac{\rho_{i,j}}{M_{i,j}})}\right)\\
&+\sum_{i=1}^{N_{x}}\sum_{j=1}^{N_{y}-1}h^x_i C_{i,j+1/2}\left (\log{(\frac{\rho_{i,j+1}}{M_{i,j+1}})}-\log{( \frac{\rho_{i,j}}{M_{i,j}})}\right) \geq0.
\end{aligned}
\end{equation*}
\end{thm}
\begin{proof}The proof is similar to that of Theorem \ref{th2.1}, details are therefore omitted.  
\end{proof}

\subsection{Fully discrete scheme} Let $\rho_{i,j}^n$ approximate ${\rho_{i,j}}(t_n)$,  then (\ref{Nosemi2}) gives the following fully discrete scheme,
\begin{equation}\label{fully2}
\frac{\rho^{n+1} _{i,j}-\rho^n_{i,j}}{\tau}=\frac{C^{n,*}_{i+1/2,j}-C^{n,*}_{i-1/2,j}}{h^x_i}+\frac{C^{n,*}_{i,j+1/2}-C^{n,*}_{i,j-1/2}}{h^y_j},
\end{equation}
where 
\begin{equation*}\label{C2}
\begin{aligned}
& C^{n,*}_{i+1/2,j}=\frac{M^n_{i+1/2,j}}{h^x_{i+1/2}}\bigg(\frac{\rho^{n+1}_{i+1,j}}{M^n_{i+1,j}}-    \frac{\rho^{n+1}_{i,j}}{M^n_{i,j}}   \bigg), \quad i=1,\cdots,N_x-1, j=1,\cdots,N_y,\\
& C^{n,*}_{i,j+1/2}=\frac{M^n_{i,j+1/2}}{h^y_{j+1/2}}\bigg(\frac{\rho^{n+1}_{i,j+1}}{M^n_{i,j+1}}-    \frac{\rho^{n+1}_{i,j}}{M^n_{i,j}}   \bigg), \quad i=1,\cdots,N_x, j=1,\cdots,N_y-1,\\
& C^{n,*}_{1/2,j}=C^{n,*}_{N_x+1/2,j}=C^{n,*}_{i,1/2}=C^{n,*}_{i,N_y+1/2}=0, \quad i=1,\cdots,N_x, j=1,\cdots,N_y,
\end{aligned}
\end{equation*}
with $M^n_{i+1/2,j}=Q_2(x_{i+1/2}, y_{j},\rho^n )$, ${M^n}_{i,j+1/2}=Q_2(x_{i}, y_{j+1/2},\rho^n )$, and $M^n_{i,j}=Q_2(x_{i}, y_{j},\rho^n )$. 

The initial data is chosen as 
\begin{equation}\label{initial2}
\rho_{i,j}^0=\frac{1}{|I_{i,j}|}\int_{I_{i,j}}\rho_0(x,y) dxdy.
\end{equation}
In 2D case, a discrete version of entropy (\ref{energy1}) may be defined as
\begin{equation}\label{energyfully2}
E_{h}^n=\sum_{i=1}^{N_x}\sum_{j=1}^{N_y}h^x_ih^y_j\left(\rho^{n}_{i, j}\log(\rho^{n}_{i, j})+V_{i, j}\rho^{n}_{i, j}+\frac{1}{2}g^{n}_{i,j} \rho^{n}_{i,j} \right),
\end{equation}
where 
$$
g_{i,j}^n=\sum_{k=1}^{N_{x}}\sum_{l=1}^{N_{y}}h^x_kh^y_lW(x_k-x_i,y_{l}-y_j)\rho^n_{k,l}.
$$
\begin{thm}
The fully discrete scheme (\ref{fully2}) has the following properties:\\
(1) Conservation of mass:
\begin{equation}\label{mass22}
 \sum_{l=1}^{N_x}\sum_{j=1}^{N_y}h^x_ih^y_j\rho_{i,j}^{n}=\int_{\Omega} \rho_0(x,y)dx dy,\ \text{ for all } n\geq 1.
\end{equation}
(2) Positivity preserving: if $\rho_{i,j}^n\geq 0$ for all $i\in\{1,\cdots,N_x\}$ and $j\in \{1,\cdots,N_y\}$, then 
$$
\rho_{i,j}^{n+1}\geq 0.
$$
(3) Entropy dissipation: there exists $\tau^*>0$ such that if $\tau\in(0,\ \tau^*)$, then
\begin{equation}\label{FE2}
E_{h}^{n+1}-E_{h}^n\leq -\frac{\tau}{2} I^n_{h},
\end{equation}
where
\begin{equation*}\label{dissipationfully2}
\begin{aligned}
I^n_{h}=&\sum_{j=1}^{N_{y}}\sum_{i=1}^{N_{x}-1}h^y_jC^{n,*}_{i+1/2,j}(\log{\frac{\rho^{n+1}_{i+1,j}}{M^n_{i+1,j}}}-\log{ \frac{\rho^{n+1}_{i,j}}{M^n_{i,j}}})\\
&+\sum_{i=1}^{N_{x}}\sum_{j=1}^{N_{y}-1}h^x_i C^{n,*}_{i,j+1/2}(\log{\frac{\rho^{n+1}_{i,j+1}}{M^n_{i,j+1}}}-\log{ \frac{\rho^{n+1}_{i,j}}{M^n_{i,j}}}) \geq0.
\end{aligned}
\end{equation*}
\end{thm}
\begin{proof}For simplicity of analysis we rewrite the scheme (\ref{fully2}) as
\begin{equation}\label{fully2equal}
\begin{aligned}
 h^x_ih^y_j\rho_{i, j}^n=&(h^x_ih^y_j M_{i,j}^n+\tau \tilde{M}^n_{i+1/2,j}+\tau \tilde{M}^n_{i-1/2,j}+\tau \tilde{M}^n_{i,j+1/2}+\tau \tilde{M}^n_{i,j-1/2})G_{i, j}^{n,*} \\
 &-\tau \tilde{M}^n_{i+1/2,j}G^{n,*}_{i+1,j} - \tau \tilde{M}^n_{i-1/2,j}G^{n,*}_{i-1,j}-\tau \tilde{M}^n_{i,j+1/2}G^{n,*}_{i,j+1}-\tau \tilde{M}^n_{i,j-1/2}G^{n,*}_{i,j-1}, 
\end{aligned}
\end{equation}

with the following notations
$$
	\tilde{M}^n_{i+1/2,j}=\frac{h^y_j}{h^x_{i+1/2}}M^n_{i+1/2,j},\quad  \tilde{M}^n_{i,j+1/2}=\frac{h^x_i}{h^y_{j+1/2}}M^n_{i,j+1/2}, \quad G^{n,*}_{i,j}=\frac{\rho^{n+1}_{i,j}}{M^n_{i,j}}.
$$
Note that the coefficient matrix of the linear system (\ref{fully2equal}) (when consider $G^{n,*}_{i,j}$ as unknowns) is strictly diagonally dominant, therefore (\ref{fully2equal}) always has a unique solution.

(1) Adding all equations in (\ref{fully2}) and using (\ref{initial2}) lead to (\ref{mass22}).

(2) Since $\rho_{i,j}^{n+1}=M_{i,j}^nG_{i,j}^{n,*}$ and $M_{i,j}^n>0$, it suffices to prove that $G_{k,l}^{n,*}=\min_{\{i,j\}}G_{i,j}^{n,*}\geq 0,$ the corresponding equation is
\begin{align*}
 h^x_kh^y_l\rho_{k,l}^n=&( h^x_kh^y_lM_{k,l}^n+\tau \tilde{M}^n_{k+1/2,l}+\tau \tilde{M}^n_{k-1/2,l}+\tau \tilde{M}^n_{k,l+1/2}+\tau \tilde{M}^n_{k,l-1/2})G_{k,l}^{n,*} \\
 &-\tau \tilde{M}^n_{k+1/2,l}G^{n+1}_{k+1,l} - \tau \tilde{M}^n_{k-1/2,l}G^{n+1}_{k-1,l}-\tau \tilde{M}^n_{k,l+1/2}G^{n+1}_{k,l+1}-\tau \tilde{M}^n_{k,l-1/2}G^{n,*}_{k,l-1}\\
 \leq &  h^x_kh^y_lM_{k,l}^nG_{k,l}^{n,*},
\end{align*}
therefore $G_{k,l}^{n,*}\geq0.$

(3) A direct calculation using (\ref{energyfully2}) gives
\begin{equation}\label{eq21}
\begin{aligned}
E_{h}^{n+1}-E_{h}^n =&\sum_{i=1}^{N_x}\sum_{j=1}^{N_y}h^x_ih^y_j(\rho^{n+1}_{i,j}\log(\rho^{n+1}_{i,j})-\rho^{n}_{i,j}\log(\rho^{n+1}_{i,j})+\rho^{n}_{i,j}\log(\frac{\rho^{n+1}_{i,j}}{\rho^{n}_{i,j}})\\
&+V_j\rho^{n+1}_{i,j}+\frac{1}{2}g^{n+1}_{i,j}\rho^{n+1}_{i,j}-V_{i,j}\rho^{n}_{i,j}-\frac{1}{2}g^{n}_{i,j}\rho^{n}_{i,j} )\\
\leq& \sum_{i=1}^{N_x}\sum_{j=1}^{N_y}h^x_ih^y_j(\log(G_{i,j}^{n,*})(\rho^{n+1}_{i,j}-\rho^{n}_{i,j})+   \frac{1}{2}g_{i,j}^n\rho_{i,j}^n-g_j^n\rho_{i,j}^{n+1}+\frac{1}{2}g_{i,j}^{n+1}\rho_{i,j}^{n+1}) ,
\end{aligned}
\end{equation}
where we have used $\log(x)\leq x-1$ and mass conservation property. By the symmetrical property of $W(x,y)$  we have
$$
\sum_{i=1}^{N_x}\sum_{j=1}^{N_y}h^x_ih^y_jg_{i,j}^n\rho_{i,j}^{n+1} =\sum_{i=1}^{N_x}\sum_{j=1}^{N_y}h^x_ih^y_jg_{i,j}^{n+1}\rho_{i,j}^{n},
$$
so that 
\begin{align*}
&\sum_{i=1}^{N_x}\sum_{j=1}^{N_y}h^x_ih^y_j(\frac{1}{2}g_{i,j}^n\rho_{i,j}^n-g_{i,j}^n\rho_{i,j}^{n+1}+\frac{1}{2}g_{i,j}^{n+1}\rho_{i,j}^{n+1})\\
&=\frac{1}{2}\sum_{i=1}^{N_x}\sum_{j=1}^{N_y}h^x_ih^y_j(g_{i,j}^{n+1}-g_{i,j}^n)(\rho_{i,j}^{n+1}-\rho_{i,j}^{n})\\
&=\frac{1}{2}\sum_{i=1}^{N_x}\sum_{j=1}^{N_y}h^x_ih^y_j(\sum_{k=1}^{N_x}\sum_{l=1}^{N_y}h^x_kh^y_lW(x_i-x_k,y_j-y_l)(\rho_{k,l}^{n+1}-\rho_{k,l}^{n}))(\rho_{i,j}^{n+1}-\rho_{i,j}^{n})\\
&\leq \frac{||W||_{\infty}}{2} \left( \sum_{i=1}^{N_x}\sum_{j=1}^{N_y}h^x_ih^y_j|\rho_{i,j}^{n+1}-\rho_{i,j}^{n}| \right)^2\\
&\leq \frac{||W||_{\infty}|\Omega|}{2}\sum_{i=1}^{N_x}\sum_{j=1}^{N_y}h^x_ih^y_j(\rho_{i,j}^{n+1}-\rho_{i,j}^{n})^2,
\end{align*}
where  $|\Omega|=\sum_{i=1}^{N_x}\sum_{j=1}^{N_y}h^x_ih^y_j$. 
Substitution of the above inequality into (\ref{eq21}) yields 
\begin{equation*}\label{eq42}
\begin{aligned}
E_{h}^{n+1}-E_{h}^n &\leq \sum_{i=1}^{N_x}\sum_{j=1}^{N_y}h^x_ih^y_j\log(G_{i,j}^{n,*})(\rho^{n+1}_{i,j}-\rho^{n}_{i,j})+\frac{||W||_{\infty}|\Omega|}{2}\sum_{i=1}^{N_x}\sum_{j=1}^{N_y}h^x_ih^y_j (\rho^{n+1}_{i,j}-\rho^{n}_{i,j})^2\\
& :=F^n_1+F^n_2. 
\end{aligned}
\end{equation*}
We proceed using summation by parts and boundary conditions so that 
\begin{align*}
F_1^n=&\tau \sum_{i=1}^{N_x}\sum_{j=1}^{N_y}\log(G_{i,j}^{n,*})(\tilde{M}^n_{i+1/2,j}(G^{n,*}_{i+1,j}-G^{n,*}_{i,j})- \tilde{M}^n_{i-1/2,j}(G^{n,*}_{i,j}-G^{n,*}_{i-1,j}))\\
&+\tau \sum_{i=1}^{N_x}\sum_{j=1}^{N_y}\log(G_{i,j}^{n,*})(\tilde{M}^n_{i,j+1/2}(G^{n,*}_{i,j+1}-G^{n,*}_{i,j})- \tilde{M}^n_{i,j-1/2}(G^{n,*}_{i,j}-G^{n,*}_{i,j-1}))\\
=&-\tau \sum_{i=1}^{N_x-1}\sum_{j=1}^{N_y}\tilde{M}^n_{i+1/2,j}(\log(G_{i+1,j}^{n,*})-\log(G_{i,j}^{n,*}))(G^{n,*}_{i+1,j}-G^{n,*}_{i,j})\\
&-\tau \sum_{i=1}^{N_x}\sum_{j=1}^{N_y-1}\tilde{M}^n_{i,j+1/2}(\log(G_{i,j+1}^{n,*})-\log(G_{i,j}^{n,*}))(G^{n,*}_{i,j+1}-G^{n,*}_{i,j})\\
=&- \tau I^n_{h}.
\end{align*}
It remains to figure out a condition on $\tau$ so that $F_2^n+\frac{1}{2}F_1^n\leq 0.$  Let $\vec{\xi}, \vec{\eta}\in \R^{N_xN_y}$ be vectors defined as:
$$\vec{\xi}=\left(\frac{\sqrt{h^x_{1}h^y_1}(\rho_{1,1}^{n+1}-\rho_{1,1}^n)}{\tau},\cdots, \frac{\sqrt{h^x_{N_x}h_1^y}(\rho_{N_x,1}^{n+1}-\rho_{N_x,1}^n)}{\tau},\cdots, \frac{\sqrt{h^x_{N_x}h^y_{N_y}}(\rho_{N_x,N_y}^{n+1}-\rho_{N_x,N_y}^n)}{\tau}\right)^T$$
$$\vec{\eta}=(\sqrt{h^x_{1}h^y_1}\log(G_{1,1}^{n,*}),\cdots,\sqrt{h^x_{N_x}h_1^y}\log(G_{N_x,1}^{n,*}), \cdots,  \sqrt{h^x_{N_x}h^y_{N_y}}\log(G_{N_x,N_y}^{n,*}))^T,$$
then $F_2^n+\frac{1}{2}F_1^n\leq 0$ if
$$
\tau^2||W||_{\infty}|\Omega|\vec{\xi}|^2+\tau \vec{\xi}\cdot \vec{\eta}\leq0.
$$
In similar manner as in 1D case, we can show that $\vec{\xi}\cdot \vec{\eta}=0$ if and only if $\vec{\xi}=0.$ 
Therefore 
$$
0< c_0 \leq \frac{-\vec{\xi}\cdot \vec{\eta}}{|\vec{\xi}|^2}\leq \frac{|\eta|}{|\xi|} \quad \text{for} \; \xi \not=0,
$$
where $c_0$ may depend on numerical solutions at $t_n$ and $t_{n+1}$. 
We thus obtain the desired result (\ref{FE2}) by taking 
$
\tau\leq \tau^*=\frac{c_0}{||W||_{\infty}|\Omega|}.$
\end{proof}
\begin{rem}
The schemes presented so far apply well to the general class of nonlinear nonlocal equations (\ref{e2}), based on the reformulation 
$$
\partial_{t}\rho=\nabla\cdot(M\nabla \frac{\rho}{M}),
$$
where $M=\rho e^{-H'(\rho)-V(\mathbf{x})-W*\rho}$ for $\rho$ away from zero. The numerical solution may be oscillatory at low density, for which one could use either upwind numerical fluxes or non-oscillatory limiters as a remedy \cite{FiniteV}. Note that for the aggregation equation (in the absence of diffusion), particle methods have been developed in \cite{CB16, PCCC18}; Particle methods naturally conserve mass and positivity, yet a large number of particles is often required to resolve finer properties of solutions.
\end{rem}

\section{Second order in-time discretization}
The numerical schemes presented so far are only first order in time.   In this section we extend these schemes with a second order in time discretization.

\subsection{Second order scheme for 1D problem}
We replace (\ref{fully}) by a two step scheme
\begin{subequations}\label{PC1}
\begin{align}
\label{fullyPR+}
&\frac{\rho^{*}_{j}-\rho^n_{j}}{\tau/2} =\frac{C^*_{j+1/2}-C^*_{j-1/2}}{h_j}, \quad  j=1,2,\cdots, N,\\
\label{fullyCR+}
& \rho_{j}^{n+1}=2\rho_{j}^*-\rho_{j}^n,\quad  j=1,2,\cdots, N,
\end{align}
\end{subequations}
where 
\begin{equation*}\label{FL}
\begin{aligned}
& C^*_{j+1/2}= \frac{M^*_{j+\frac{1}{2}}}{h_{j+1/2}}(\frac{\rho^*_{j+1}}{M^*_{j+1}}-\frac{\rho^*_j}{M^*_j}), \quad \text{for } \ j=1,2,\cdots, N-1,\\ 
& C^*_{1/2}=0,\quad C^*_{N+1/2}=0,
\end{aligned}
\end{equation*}
with $M^*_{j+1/2}=Q_1(x_{j+1/2}, \frac{3}{2}\rho^n-\frac{1}{2}\rho^{n-1})$ and $M^*_{j}=Q_1(x_{j}, \frac{3}{2}\rho^n-\frac{1}{2}\rho^{n-1})$.
The scheme (\ref{PC1}) has following properties.
\begin{thm}
Let $\rho^{n+1}_j$ be obtained from (\ref{PC1}), then 

(1) Conservation of mass:
\begin{equation*}\label{mass2}
\sum_{j=1}^Nh_j\rho_j^{n}= \int_{\Omega} \rho_0(x) dx, \ \ \text{ for } n\geq 1.
\end{equation*}

(2) Positivity preserving: if $\rho_j^n\geq0$ for all $j=1,\cdots, N,$ then
$$\rho_j^{n+1}\geq 0, \quad j=1,\cdots,N,$$

provided $\tau$ is sufficiently small.
\end{thm}
\begin{proof}
(1) From the scheme construction, the conservation property remains hold.  

(2) 
Setting 
$$
G^{n}_j=\frac{\rho^n_j}{M^{*}_j}, \quad g^*_{j+1/2}=\frac{M^*_{j+1/2}}{h_{j+1/2}},
$$
and a careful regrouping leads to the following linear system  
\begin{equation}\label{G}
\begin{aligned}
& \left(  M^*_{1}+ \frac{\tau}{2h_1}g^*_{3/2}\right)G^{n+1}_1 -\frac{\tau}{2h_1}g^*_{3/2}G^{n+1}_{2}=b_1,\\
& \left( M^*_{ j}+ \frac{\tau}{2h_j}(g^*_{j+1/2}+g^*_{j-1/2})\right)G^{n+1}_j -\frac{\tau}{2h_j}g^*_{j+1/2}G^{n+1}_{j+1}- \frac{\tau}{2h_j}g^*_{j-1/2}G^{n+1}_{j-1}=b_j,\\
& \left( M^*_{ N}+ \frac{\tau}{2h_N}g^*_{N-1/2}\right)G^{n+1}_N - \frac{\tau}{2h_N}g^*_{N-1/2}G^{n+1}_{N-1}=b_{N},
\end{aligned}
\end{equation}
where $j=1, \cdots, N-1$,  with the right hand side vector given by
\begin{equation*}
\begin{aligned}
& b_1=\left(M^*_{ 1}- \frac{\tau}{2h_1}g^*_{3/2}\right)G^{n}_1+\frac{\tau}{2h_1}g^*_{3/2}G^{n}_{2},\\
&b_j= \left( M^*_{ j}- \frac{\tau}{2h_j}(g^*_{j+1/2}+g^*_{j-1/2})\right)G^{n}_j +\frac{\tau}{2h_j}g^*_{j+1/2}G^{n}_{j+1}+ \frac{\tau}{2h_j}g^*_{j-1/2}G^{n}_{j-1},\; j=1, \cdots, N-1, \\
&b_{N}= \left( M^*_{ N}- \frac{\tau}{2h_N}g^*_{N-1/2}\right)G^{n}_N + \frac{\tau}{2h_N}g^*_{N-1/2}G^{n}_{N-1}.
\end{aligned}
\end{equation*}
The linear system (\ref{G})  admits a unique solution $\{G^{n+1}_j\}$ since its coefficient matrix is strictly diagonally dominant. 
Following the proof of (2) in Theorem \ref{thm3.1}, we see that $G^{n+1}_j\geq0$ is ensured if each $b_j \geq 0$, which is the case provided 
$$
\tau  \leq \min \left\{\frac{2h_1 M_{ 1}^*}{g^*_{ 3/2} },\quad  \min_{1< j< N} \frac{2h_j M_{ j}^*}{g^*_{ j+1/2} +g^*_{ j-1/2}}, \quad \frac{2h_N M_{ N}^*}{g^*_{ N-1/2}} \right\}.
$$
The stated result thus follows. 
\end{proof}

For large time step $\tau$, non-negativity of $\rho_{ j}^{n+1}$ obtained by the second order scheme (\ref{PC1}) may not  be guaranteed, we introduce a local limiter to resolve the solution positivity. 

\subsection{Local limiter and algorithm}
We begin to design a local limiter to restore positivity of  $\{c_j\}_{j=1}^N$ if  $\sum_{j=1}^Nc_j>0$, but $c_k<0$ for some $k$.  The idea is to find a neighboring index set $ S_k$ such that the local average 
$$
\bar c_k= \frac{1}{|S_k|} \sum_{j\in S_k} c_j>0, 
$$
where $|S_k|$ denotes  the minimum number of indexes for which $c_j\not=0 \; \text{and} \;  \bar c_k>0$,  then use this  as a reference to define the following scaling limiter,
\begin{equation}\label{RC}
\tilde{c}_{j}=\theta c_{j}+(1-\theta) \bar{c}_k, \quad j\in S_k,
\end{equation}
where
\begin{equation*}\label{TH}
\theta=\min\left\{1,  \frac{\bar{c}_k}{\bar{c}_k-c_{min}}\right\}, \quad c_{min}=\min_{j\in S_k} c_{j}.
\end{equation*}

\begin{lem} This limiter has the following properties: \\
(1) $\tilde{c}_{j}\geq 0$ for all $j\in S_k$, \\
(2) $\sum_{j\in S_k }\tilde{c}_{j}=\sum_{j\in S_k} c_{j}$, and \\
(3) $|\tilde{c}_j-c_j|\leq |S_k| (-\min_{j\in S_k} c_{j})$.
\end{lem} 
\begin{proof}
(1) This follows from the definition of $\theta$ and (\ref{RC}).

(2) By (\ref{RC}) and the definition of $\bar{c}_k$, it follows that 
\begin{align*}
\sum_{j\in S_k}\tilde{c}_{j}=  \theta |S_k| \bar c_k + (1-\theta) \bar{c}_k |S_k|= \sum_{j\in S_k}c_{j}.       
\end{align*}
(3)   From (\ref{RC}) it follows that for all $j\in S_k$, 
\begin{equation*}\label{AC1}
\begin{aligned}
|\tilde{c}_j-c_j| & =(1-\theta)|\bar{c}_k-c_j|   = -c_{min} \frac{|\bar{c}_k-c_j|}{(\bar{c}_k-c_{min})}\\
& \leq (-c_{min}) \max \left\{ 1,\frac{c_{max}-\bar{c}_k}{\bar{c}_k-c_{min}}\right\},
\end{aligned}
\end{equation*}
where  $c_{max}:=\max_{j\in S_k}c_{j}$ and $c_{min}:=\min_{j\in S_k}c_{j}$. 
 Note that $\sum_{j\in {S_k}}(\bar{c}_k-c_j)=0$ implies  
$$
\sum_{j\in {S_k^+}}(c_j-\bar{c}_k)= \sum_{j\in {S_k^-}}(\bar{c}_k-c_j),
 $$
 in which each term involved on both sides is nonnegative.  Hence, 
 $
 c_{max}-\bar{c}_k\leq |{S_k}|(\bar{c}_k-c_{min}).
 $
 Obviously, $|S_k|\geq 1$. Hence the claimed bound follows.  
\end{proof}
\begin{rem}  In general,  $|S_k|$ may not be bounded.  For instance, we let 
$$
c_j=\frac{1}{2^j} \text{ for } j=1,\cdots, N-1, \text{ and } c_{N}=-\frac{1}{2},
$$
then $\sum_{j=1}^Nc_j=\frac{1}{2}-\frac{1}{2^{N-1}}>0$, but $\sum_{j=2}^Nc_j=-\frac{1}{2^{N-1}}<0$. This implies that  $|S_N|=N$ since  $S_N=\{1,\cdots, N\}.$
\end{rem}
The above limiter when applied to $\{\rho_j\}$ with $c_j=h_j\rho_j$ 
gives
\begin{align}\label{tcj}
\tilde \rho_j= \theta \rho_j +(1-\theta) \frac{\bar c_k}{ h_j}, 
\end{align}
where
$$
\theta=\min \left\{1, \ \frac{\bar{c}_k}{\bar{c}_k-c_{min}} \right\}, \quad c_{min}=\min_{j\in S_k}h_j\rho_{j}, \quad \bar c_k =\frac{1}{|S_k|}\sum_{j\in S_k}h_j\rho_j .
$$
Such limiter still respects the local mass conservation. In addition, for any sequence $g_j$ with $g_j\geq 0$, 
 we have 
$$
|\tilde \rho_j -g_j|\leq (1+|S_k|\alpha)\max_{j \in S_k}|\rho_j -g_j| , \quad j\in S_k,
$$
where $\alpha$ is the upper bound of mesh ratio $h_i/h_j$.  Let $\rho_j$ be the approximation of $\rho(x)\geq 0$, we let $g_j=\rho(x_j)$ or the average of $\rho$ on $I_j$, so we can assert that the accuracy is not destroyed by the limiter as long as $|S_k|\alpha$ is uniformly bounded. 
In practice, it is indeed the case as verified by our numerical tests when using shape-regular meshes. 


Indeed,  the boundedness of $|S_k|$ can be proved rigorously for shape-regular meshes.
\begin{thm}\label{thm0}
Let $\rho(x) \geq 0$, be in $C^2(\Omega)$, and $\{\rho_j\}$ be an approximation of $\rho(x)$ such that 
$
|\rho_j-\rho(x_j)|\leq C h^2,
$
where $h=\min_{1\leq j\leq N} h_j$ and $h_j\leq \alpha h$ for some $\alpha>0$.  If $\rho_k<0$ (or only finite number of neighboring values are negative), then there exists $K^*>0$ finite such that 
$$
|S_k|\leq K^*.
$$
where $K^*$ may depend on the local meshes associated with $S_k$.  
\end{thm}
\begin{proof} Under the assumption $\rho_k<0$, $\rho$ must touch zero near $x_k$.  We discuss the case where $\rho(x^*)=0$ and $\rho'(x^*)=0$ with $\rho(x)>0$ for $x>x^*$ locally with $x^*\in I_k$. The case where $\rho(x)>0$ for $x<x^*$ can be handled as well.  Without loss of generality, we consider $k=1$ with $x^*\in I_1$, and 
$\int_{I_1}\rho(x)dx >0$. It suffices to find $K$ such that 
 \begin{equation}\label{L0}
\sum_{j=1}^Kh_j \rho_j>0.
\end{equation}
Using the error bound we have 
$$
\rho_j \geq \rho(x_j)-Ch^2. 
$$
Also from $\rho \in C^2$ we can deduce that 
$$
\rho(x_j)\geq \bar \rho_j -\lambda h_j^2,
$$
with $\lambda=\frac{1}{24}\max_{x\in \Omega}|\rho''|$ and the cell average $\bar \rho_j=\frac{1}{h_j} \int_{I_j}\rho(x)dx$.   Combining these we see that the left hand side of (\ref{L0}) is bounded from below by 
\begin{align*}
\sum_{j=1}^Kh_j\rho_j \geq & \sum_{j=1}^K h_j(\bar \rho_j -C h^2 -\lambda h_j^2)\\
\geq & \int_{x_{1/2}}^{x_{K+1/2}}\rho(x)dx - (\lambda+C) \sum_{j=1}^Kh_j^3 \\
\geq & \int_{x_{1/2}}^{x_{K+1/2}}\rho (x)dx- (\lambda+C) {h}^2\alpha^2 \sum_{j=1}^Kh_j\\
=& \left[ \int_0^1 \rho\left(\theta \eta +x_{1/2}\right)d\theta -(\lambda+C) {h}^2\alpha^2 \right]\eta,
\end{align*}
where $\eta:=\sum_{j=1}^Kh_j$, and we have used $h_j \leq h\alpha$. Using the fact $Kh \leq \eta $, the term in the bracket is bounded below by  
$$
\int_0^1 \rho\left(\theta \eta +x_{1/2}\right)d\theta -(\lambda+C) \eta^2 \alpha^2/K^2,
$$  
which is positive if 
$$
K> \frac{\alpha \sqrt{\lambda +C}\eta}{\sqrt{\int_0^1  \rho\left(\theta \eta +x_{1/2}\right)d\theta}}.
$$
This can be ensured if we take 
$$
K=\lfloor A \rfloor+1, 
$$
where for $\Omega=[a, b]$, 
$$
A=\max_{z\in [h_1, b-a]}\frac{\alpha \sqrt{\lambda +C}z}{\sqrt{\int_0^1  \rho\left(\theta z +x_{1/2}\right)d\theta}}
$$
which is bounded and depends on $h_1$. For general cases a different bound can be identified and it may depend 
on local meshes. 
\end{proof}
Note that our numerical solutions feature the following property:  if $\rho_j^n=0$, then $\rho_j^{n+1}=2\rho^*_j-\rho_j^n\geq 0$ due to the fact that $\rho_j^*\geq 0$ for all $j=1, \cdots, N$. This means that if $\rho_0(x)=0$ 
on an interval, then $\rho_j^1$ cannot be negative in most of nearby cells.  Thus negative values appear only where the exact solution turns from zero to a positive value, and the number of these values are finitely many. Our result in Theorem \ref{thm0} is thus applicable.  \\

\noindent {\bf Algorithm.} We have the following algorithm:
\begin{enumerate}
\item Initialization: From initial data $\rho_0(x)$, obtain $\rho^0_{j}=\frac{1}{h_j}\int_{I_j}\rho_0(x)dx, \ j=1,\cdots,N,$ by using a second order quadrature.
\item Update to get $\{\rho^1_{j}\}$ by the first order scheme (\ref{fully}).
\item Marching from $\{\rho^n_{j}\}$ to $\{\rho_{j}^{n+1}\}$ for $n=1, 2, \cdots,$ based on (\ref{PC1}).
\item Reconstruction: if necessary, locally replace $ \rho_{j}^{n+1}$ by $\tilde \rho_{j}^{n+1}$ using the limiter defined in (\ref{tcj}). 
\end{enumerate}

The following algorithm can be called to find an admissible set $S_k$  used in (\ref{tcj}).
\begin{enumerate}
\item[(i)] Start with $S_k=\{k\}$, $m=1$.
\item[(ii)]  If $k-m \geq  1$ and $c_{k-m}\ne 0$, then set $S_k=S_k\cup \left\{k-m\right\}$. \\If $\bar{c}_k>0$, then stop, else go to (iii).
\item[(iii)] If $k+m\leq N$ and $c_{k+m}\ne 0$, then set $S_k=S_k\cup \{ k+m\}$.\\
 If $\bar{c}_k>0$, then stop, else set $m=m+1$ and go to (ii).
\end{enumerate}

\subsection{Second order scheme for 2D problem}
A similar two step time-discretization technique can be applied to higher dimensional problems. 
In the 2D case, that with scheme (\ref{Nosemi2}) gives the following fully discrete scheme,
\begin{subequations}\label{PC2}
\begin{align}
\label{PR+2}
\frac{\rho^* _{i,j}-\rho^n _{i,j}}{\tau/2}&=\frac{C^*_{i+1/2,j}-C^*_{i-1/2,j}}{h^x_i}+\frac{C^*_{i,j+1/2}-C^*_{i,j-1/2}}{h^y_j},\\
\label{CR+2}
\rho^{n+1}_{i,j}&=2\rho^*_{i,j}-\rho^n_{i,j},
\end{align}
\end{subequations}
where 
\begin{equation*}\label{Cf2}
\begin{aligned}
& C^*_{i+1/2,j}=\frac{M^*_{i+1/2,j}}{h^x_{i+1/2}}\bigg(\frac{\rho^*_{i+1,j}}{M^*_{i+1,j}}-    \frac{\rho^*_{i,j}}{M^*_{i,j}}   \bigg), \quad i=1,\cdots,N_x-1, j=1,\cdots,N_y,\\
& C^*_{i,j+1/2}=\frac{M^*_{i,j+1/2}}{h^y_{j+1/2}}\bigg(\frac{\rho^*_{i,j+1}}{M^*_{i,j+1}}-    \frac{\rho^*_{i,j}}{M^*_{i,j}}   \bigg), \quad i=1,\cdots,N_x, j=1,\cdots,N_y-1,\\
& C^*_{1/2,j}=C^*_{N_x+1/2,j}=C^*_{i,1/2}=C^*_{i,N_y+1/2}=0, \quad i=1,\cdots,N_x, j=1,\cdots,N_y,
\end{aligned}
\end{equation*}
with $M^{*}_{i+1/2,j}=Q_2(x_{i+1/2}, y_{j},\frac{3}{2}\rho^n-\frac{1}{2}\rho^{n-1} )$, ${M^*}_{i,j+1/2}=Q_2(x_{i}, y_{j+1/2},\frac{3}{2}\rho^n-\frac{1}{2}\rho^{n-1} )$, and $M^*_{i,j}=Q_2(x_{i}, y_{j}, \frac{3}{2}\rho^n-\frac{1}{2}\rho^{n-1} )$. 
In an entirely similar fashion (details are therefore omitted), we can prove the following.
\begin{thm}
The fully discrete scheme (\ref{PC2}) has the following properties:\\
(1) Conservation of mass:
$$
 \sum_{i=1}^{N_x}\sum_{j=1}^{N_y}h^x_ih^y_j\rho_{i,j}^{n}=\int_{\Omega} \rho_0(x,y) dx dy, \text{ for } n\geq 1.
$$
(2) Positivity preserving: if $\rho_{i,j}^n\geq 0$ for all $i\in\{1,\cdots,N_x\}$ and $j\in \{1,\cdots,N_y\}$, then 
$$
\rho_{i,j}^{n+1}\geq 0,
$$
provided $\tau$ is sufficiently small.
\end{thm}
\subsection{Local limiter and algorithm} If the time step $\tau$ is not small, positivity of $\rho_{i, j}^n$ is not guaranteed for $n\geq 2$.
We use the following limiter to resolve this issue: 
\begin{align}\label{tcj2}
\tilde \rho_{i,j}= \theta \rho_{i,j} +(1-\theta) \frac{\bar c_{k, l}}{ h^x_ih^y_j}, 
\end{align}
with
$$
\theta=\min\left\{1, \ \frac{\bar{c}_{k, l}}{\bar{c}_{k,l}-c_{min}} \right\}, \quad c_{min}=\min_{(i,j)\in S_{k,l}}h^x_ih^y_j\rho_{i,j}, \quad \bar c_{k, l}=\frac{1}{|S_{k,l}|}\sum_{{(i,j)}\in S_{k,l}}h^x_ih^y_j\rho_{i,j},
$$
where $S_{k,l}$ denotes the minimum number of indexes for which  $\rho_{i, j}\not=0$ and  $\bar c_{k, l}>0$.

The limiter (\ref{tcj2}) can be shown to be nonnegative and satisfy the local mass conservation. 
In addition, for any $g_{i, j} \geq 0$ we have 
$$
|\tilde \rho_{i,j} -g_{i, j}|\leq (1+|S_{k,l}| \alpha)\max_{(i,j) \in S_{k,l}}|\rho_{i,j}-g_{i, j}| , \quad (i,j)\in S_{k,l},
$$
where $\alpha$ is the upper bound of 2D mesh ratios. Hence the second order accuracy remains 
for shape-regular meshes since $|S_{k,l}|$ can be shown bounded as in the one-dimensional case. \\

\noindent{\bf Algorithm}
Our algorithm for 2D problem is given as follows:
\begin{enumerate}
\item Initialization: From initial data $\rho_0(x,y)$, obtain $\rho^0_{i,j}=\frac{1}{I_{i,j}}\int_{I_{i,j}}\rho_0(x,y)dxdy, \ i=1, \cdots, N_x,\ j=1,\cdots,N_y,$ by using a second order quadrature.
\item Update to get $\{\rho^1_{i,j}\}$ by the first order scheme (\ref{fully2}).
\item March from $\{\rho^n_{i,j}\}$ to $\{\rho_{i,j}^{n+1}\}$ based on the scheme (\ref{PC2}).
\item Reconstruction: if necessary, locally replace $\rho_{i,j}^{n+1}$ by $\tilde \rho_{i,j}^{n+1}$ using the limiter defined in (\ref{tcj2}). 
\end{enumerate}

The following algorithm can be called to find an admissible set $S_{k,l}$  used in (\ref{tcj2}).
\begin{enumerate}
\item[(i)] Start with $S_{k,l}=\{(k,l)\}$, $m=1$.
\item[(ii)]  For $d_y=\max \{1, l-m\}: \min \{l+m, N_y\}$ and  $d_x=\max \{1, k-m\}:\min \{k+m, N_x\}$,\\ 
If $(d_x,d_y)\notin S$ and $c_{k-m}\ne 0$, then set $S_{k,l}=S_{k,l}\cup \left\{ (d_x,d_y) \right\}$.\\
 If $\bar{c}_{k,l}>0$, then stop, else go to (iii).
\item[(iii)] Set $m=m+1$ and go to (ii).
\end{enumerate}
\section{Numerical Examples}
In this section, we implement the fully discrete schemes (\ref{fully}) and (\ref{fully2}) and second order extensions (\ref{PC1}) and (\ref{PC2}). Errors in 1-D case are measured in the following discrete norms:  
$$e_{l^1}=h\sum_{i=1}^N|\rho_i^n-\bar{\rho}_i^n|,$$
$$e_{l^{\infty}}=\max_{1\leq i\leq N}|\rho_i^n-\bar{\rho}_i^n|.$$
Here $\bar{\rho}_i^n$ is cell average of the exact solution on $I_i$ at time $t=n\tau.$

\subsection{One-dimensional tests}
\begin{example}\label{ex51}
 (Accuracy test) In this example we test the accuracy of scheme (\ref{fully}) and scheme (\ref{PC1}) Consider the initial value problem with source term
 \begin{equation}
 \left \{
\begin{array}{rl}
 \hfill  \partial_t \rho =&\partial_x(\partial_x \rho +  \rho \partial_x(V(x)+W*\rho))+F(x,t),  \hfill \ \ \   t>0, \ x\in [-\pi, \ \pi],\\
  \hfill \rho(x,0)=&2+\cos(x)   ,  \hfill \ \ \ \ x\in [-\pi, \ \pi],
\end{array}
\right.
\end{equation}
subject to zero flux boundary conditions.  Here we take $V(x)=\cos(x), W(x)=\cos(x),$  and 
$$
F(x,t)=\pi e^{-2t}(2\cos^2(x)+2\cos(x)-1)+e^{-t}(2\cos^2(x)+2\cos(x)-3).
$$ 
One can check that the exact solution to (\ref{ex51}) is 
$$
\rho(x,t)=e^{-t}(2+cos(x)).
$$ 
 We compute to $t=1$, first use time step $\tau=0.1h $ and $\tau=h^2$ to check accuracy of scheme (\ref{fully}),  then use $\tau=h$ to check accuracy of scheme (\ref{PC1}), results are reported in Table 1 and Table 2 respectively. We see that the scheme (\ref{fully}) is first order accurate in time and second order accurate in space, while the scheme  (\ref{PC1}) is second order accurate both in time and space.

Note that the exact solution is $\rho(x,t)=e^{-t}(2+cos(x))$, which is far above $0$ for $t\in [0, 1]$. Hence the positivity-preserving limiter is not activated in this test.

\begin{table}[ht]\label{T1}
        \centering
                \caption{Accuracy of scheme (\ref{fully}) with $\tau=0.1h$ and $\tau=h^2$ .}
\begin{tabular}{|l|l|l|l|l|l|l|l|l|}
\hline
 & \multicolumn{4}{l|}{errors and orders with $\tau=0.1h$} & \multicolumn{4}{l|}{errors and orders with $\tau=h^2$} \\ \hline
 N     &     $l^1$ error & order& $l^{\infty}$ error &order    &    $l^1$ error & order& $l^{\infty}$ error &order     \\ \hline
 40   &     0.70474E-01&-&0.26268E-01&-                       &     0.10451E-00&-&0.46075E-01&-    \\ 
80 &0.32212E-01&1.1295&0.15021E-01&0.8063           &    0.25847E-01&2.0156&0.11397E-01&2.0153    \\ 
 160 &0.15796E-01&1.0280&0.79593E-02&0.9163        &    0.64441E-02&2.0039&0.28433E-02&2.0030   \\ 
320 &0.78955E-02&1.0005&0.40881E-02&0.9612          &     0.16098E-02&2.0011&0.71027E-03&2.0011   \\ \hline
\end{tabular}
\end{table}
 
 \begin{table}[ht]\label{ex1h2}
        \centering
                \caption{Accuracy of scheme (\ref{PC1}) with $\tau=h$ .}
        \begin{tabular}{|c| c |c |c| c |c|}
            \hline
              N& \ $l^1$ error & order& $l^{\infty}$ error &order \\ [0.5ex] 
            \hline
            40 &0.14049E-00&-&0.43022E-01&-\\
            80 &0.35941E-01&1.9668&0.10729E-01&2.0036\\
            160 &0.90784E-02&1.9851&0.26805E-02&2.0009\\
            320 &0.22814E-02&1.9925&0.67108E-03&1.9980
             \\ [1ex]
            \hline
        \end{tabular}
     \end{table}
 \end{example}

\begin{example} \label{ex52}
In this example, we study dynamics of linear Fokker-Plank equations by considering the following problem  
\begin{equation}
\begin{array}{rl}
  \hfill \partial_t \rho =&\partial_x(\partial_x \rho +x \rho),  \hfill \ \ \   t>0, \ x\in [-5,\ 5], \\
  \end{array}
\end{equation}
with initial condition
\begin{equation}
\rho(x,0)=\left \{
\begin{array}{rl}
 &\frac{1}{7} \int_{\Omega}e^{\frac{-x^2}{2}}dx,  \hfill \ \ \   \ x\in [-3.5, \ 3.5], \\
& 0,  \hfill \ \ \   \ \text{otherwise},\\
\end{array}
\right.
\end{equation}
and zero flux boundary conditions $(\partial_x \rho +x \rho)|_{x=\pm 5}=0.$ 

This is (\ref{mainmodel11}) with $V(x)=\frac{x^2}{2}$ and $W(x)=0$. The steady state to (\ref{ex52}) is $\rho_{eq}(x)=e^{-\frac{x^2}{2}}.$  We use the time step $\tau =0.1$ to compute solutions up to $t=4$, with $N=200$. In Fig.1(a) are snap shots of $\rho$ at $t=0,\ 0.2,\ 0.5, \ 1, \ 4,$ and the steady state. 
Fig.1(b) shows the mass conservation and energy decay. We observe from this figure that the solution of problem (\ref{ex52}) becomes indistinguishable from the steady state after $t=2.$ 
Compared in Fig.2  are numerical solutions obtained by the second order scheme (\ref{PC1}) with and without the local limiter. We see that the limiter produces positive solutions and reduces solution oscillations. 

 \begin{figure}[!tbp]
\caption{First order scheme for Example \ref{ex52}. }
\centering  
\subfigure{\includegraphics[width=0.48\linewidth]{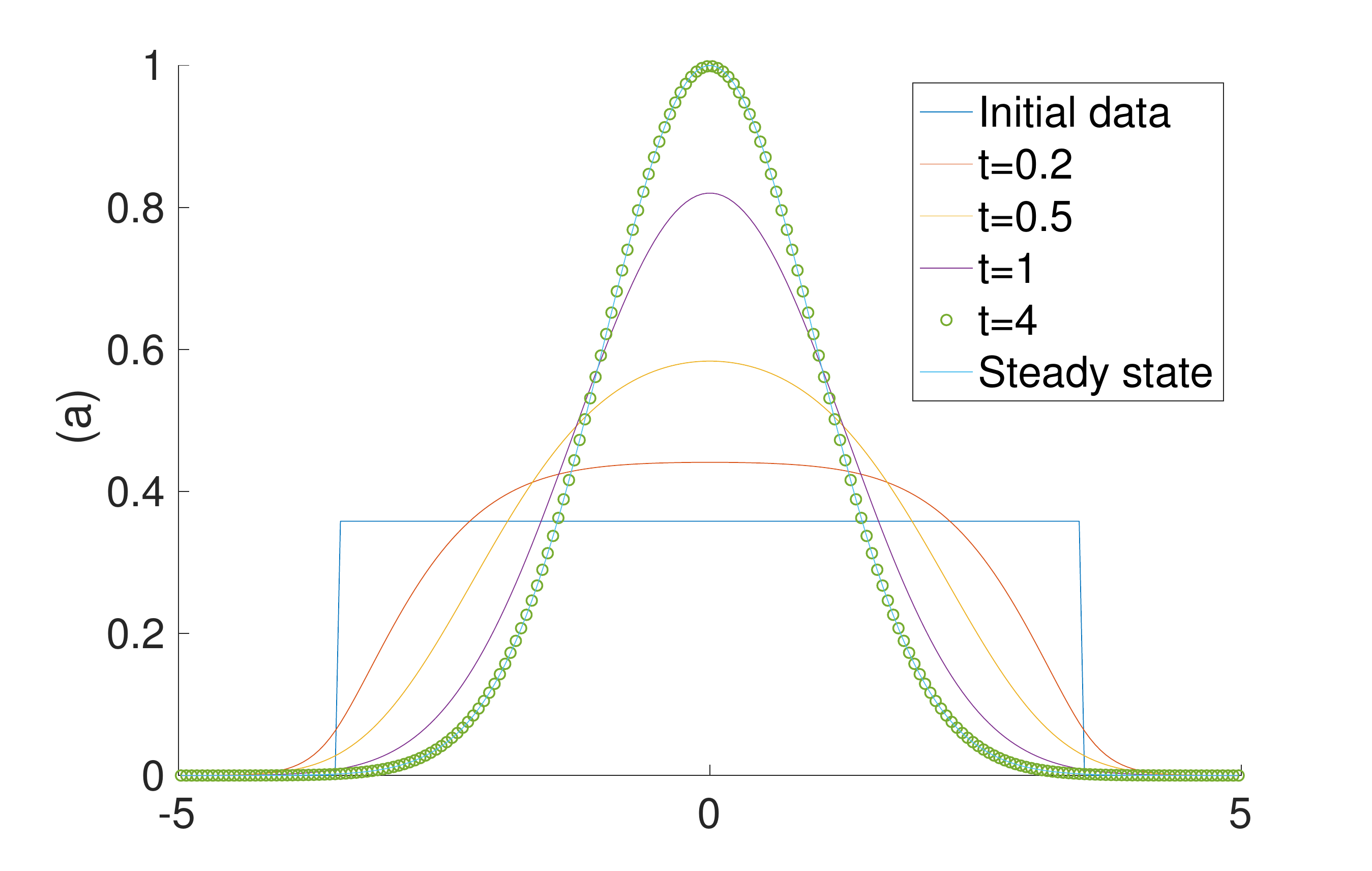}}
\subfigure{\includegraphics[width=0.48\linewidth]{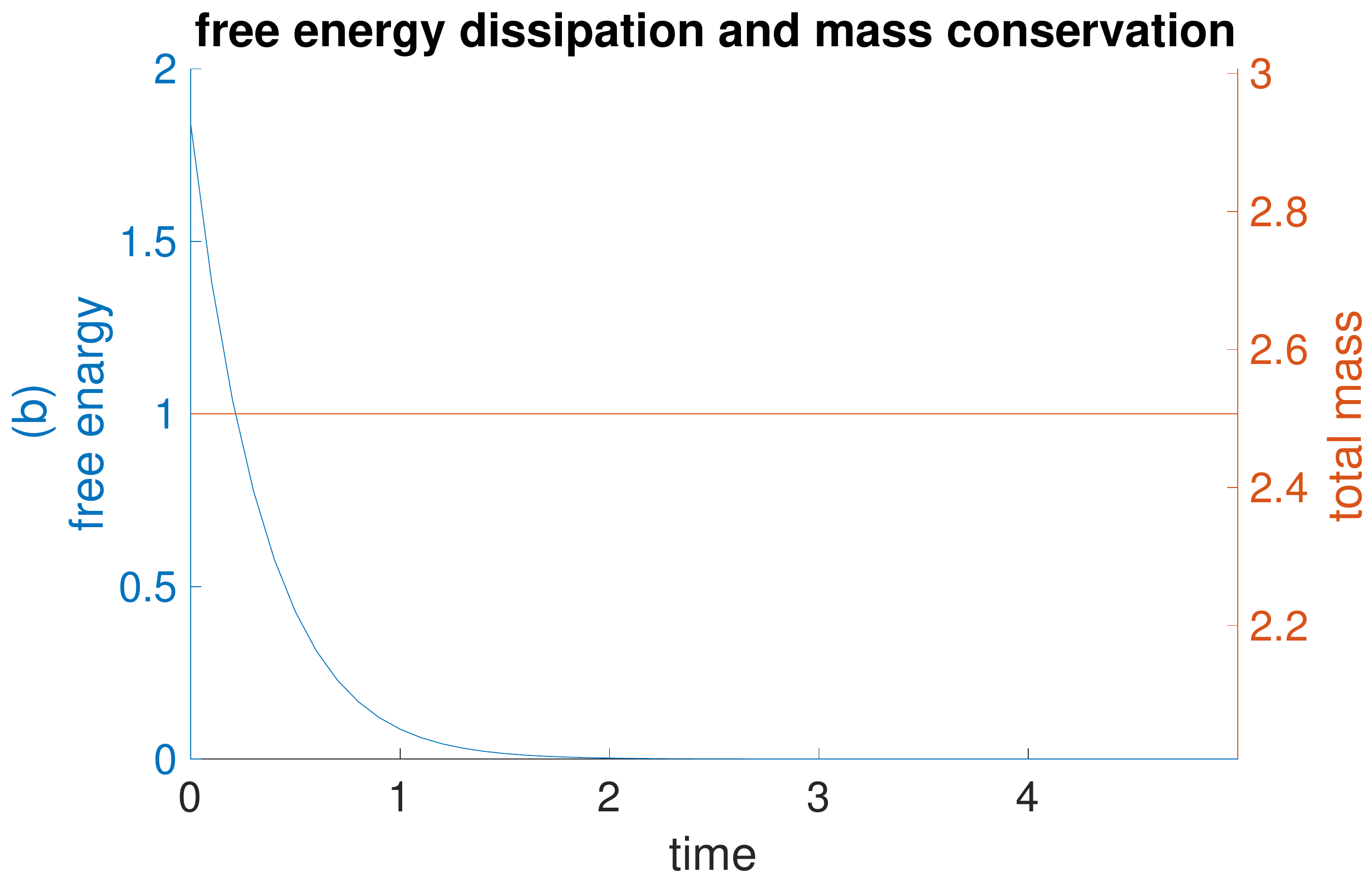}}
\end{figure}
 \begin{figure}[!tbp]
\caption{Second order scheme (with and without limiter) for Example \ref{ex52}. }
\centering  
\subfigure{\includegraphics[width=0.48\linewidth]{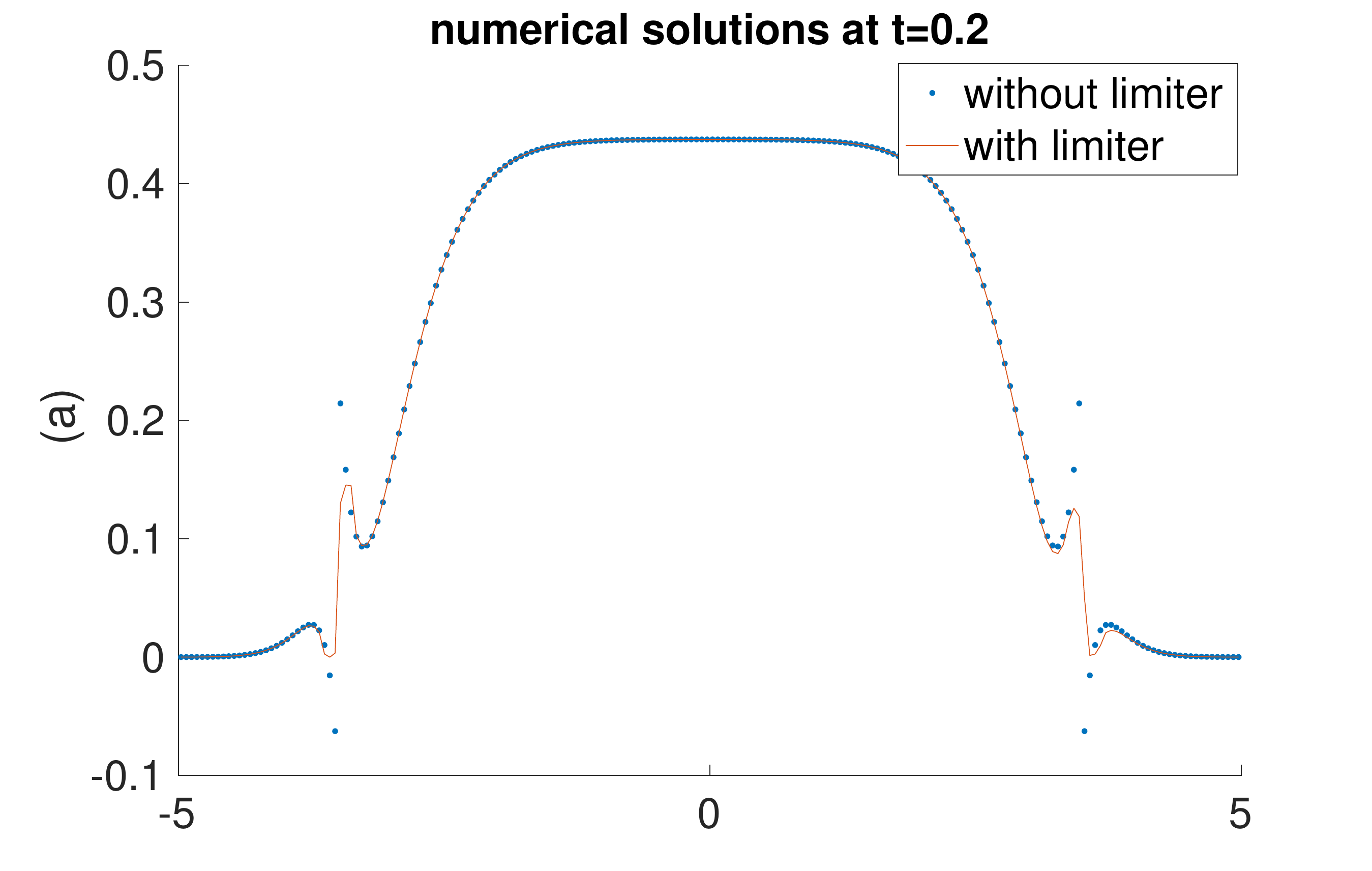}}
\subfigure{\includegraphics[width=0.48\linewidth]{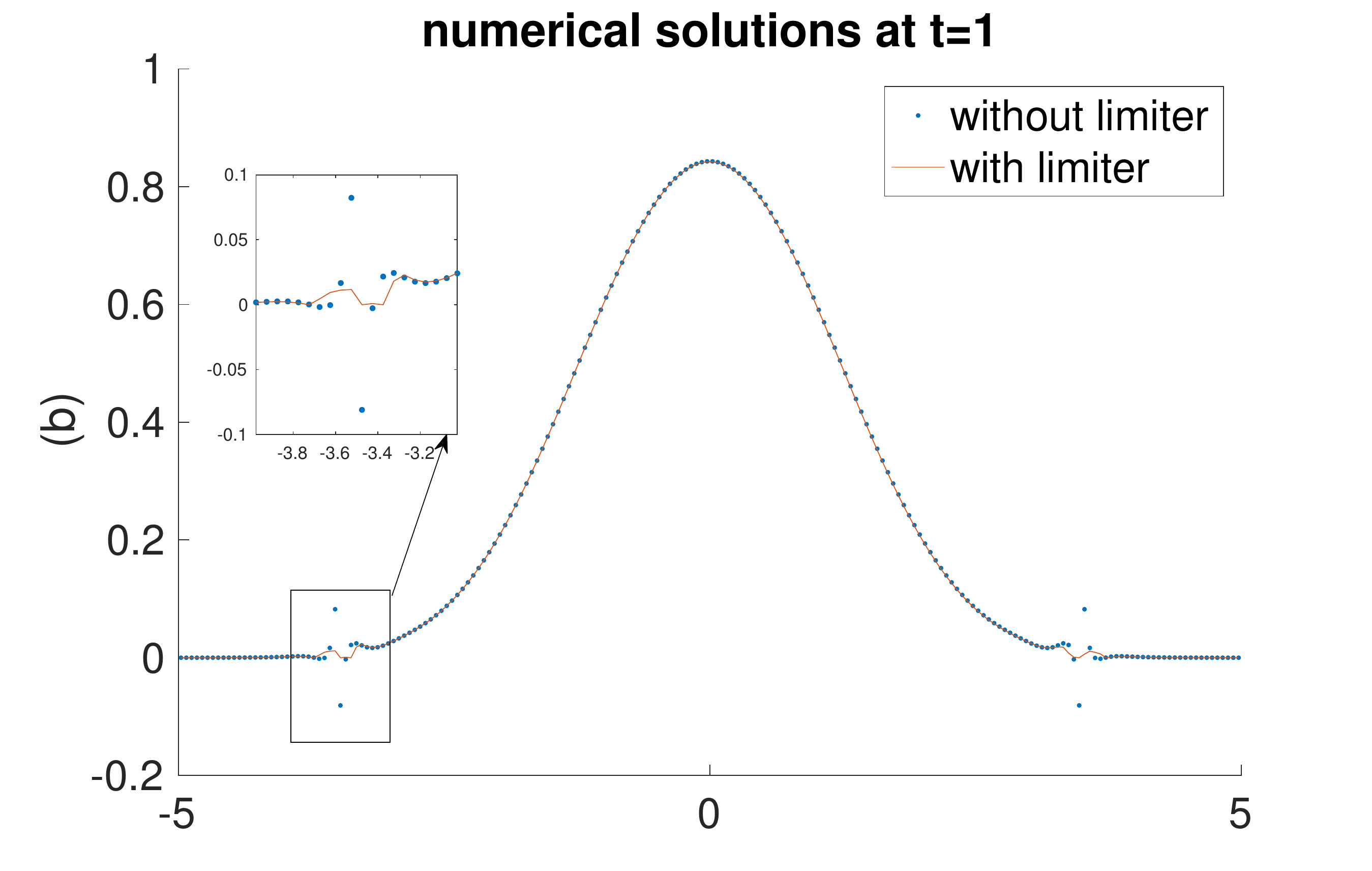}}
\end{figure}
\end{example}


\begin{example} \label{ex53}(Doi-Onsager equation with the Maier-Saupe potential) In this example, we consider the Doi-Onsager equation with Maier-Saupe potential
\begin{equation}\label{eq53}
\left \{
\begin{array}{rl}
  \hfill \partial_t \rho =&\partial_x(\partial_x\rho+  \alpha \rho\partial_x (W*\rho) )),  \hfill  \ \  W(x)=\sin^2(x) \ \   t>0, \ x\in  [0, 2\pi]\\
   \hfill \rho(x,0)=&\frac{x+1}{2\pi(\pi+1)}  ,  \hfill
 \end{array}
\right.
\end{equation}
subject to zero flux boundary conditions. Here $\alpha$ is the intensity parameter. Stationary solutions of (\ref{eq53}) have been an  interesting subject of study,  since when $\alpha$ increases, phase transition from isotropic state to nematic state will appear. A detailed characterization of solotions can be found in \cite{LiuZhang}:  for $0<\alpha\leq\alpha^*= 4$, the only stationary solution is the isotropic state  $\rho_{eq}(x)=\frac{1}{2\pi}.$ When $\alpha>\alpha^*$ besides the constant solution $\rho_{eq}(x)=\frac{1}{2\pi},$ there are other solutions given by
$$
\rho_{eq}(x)=\frac{e^{-\eta^*\cos2(x-x_0)}}{\int_0^{2\pi}e^{-\eta^*\cos(2x)}dx},
$$
where $x_0$ is arbitrary, $\eta^*>\frac{\alpha}{2}\sqrt{1-4/\alpha}$ is uniquely determined by
$$
\frac{\int_0^{2\pi}cos(2x)e^{-\eta^*\cos(2x)}dx}{\int_0^{2\pi}e^{-\eta^*\cos(2x)}dx}+\frac{2\eta}{\alpha}=0.
$$
We use scheme (\ref{fully}) and choose the time step $\tau=0.1$ to compute up to $T=30$ with $N=80.$ In Fig. 3(a) are snap shots of solutions to (\ref{eq53}) for $\alpha=3<\alpha^*$ at $t=0,\ 0.5,\ 5, \ 15, \ 25, \ 30$. Fig.3 (b) shows mass conservation and energy decay,  from which we can observe that the problem (\ref{eq53}) is already at steady state $\rho_{eq}(x)=\frac{1}{2\pi}$ after $t=20.$ In Fig. 4(a) are snap shots of solutions to (\ref{eq53}) for $\alpha=5>\alpha^*$ at $t=0,\ 0.5,\ 1, \ 5, \ 25, \ 35$. Fig.4 (b) shows mass conservation and energy decay, which tells that problem (\ref{eq53}) is at already steady state after $t=30.$  { Our method gives satisfying results for the problem, consistent with the numerical results obtained in \cite{taylor} by an explicit scheme with Euler forward time discretization. }

 \begin{figure}[!tbp]
\caption{Solution evolution and energy dissipation for Example (\ref{ex53}) with $\alpha=3.$ }
  \centering
       \subfigure{\includegraphics[width=0.5\linewidth]{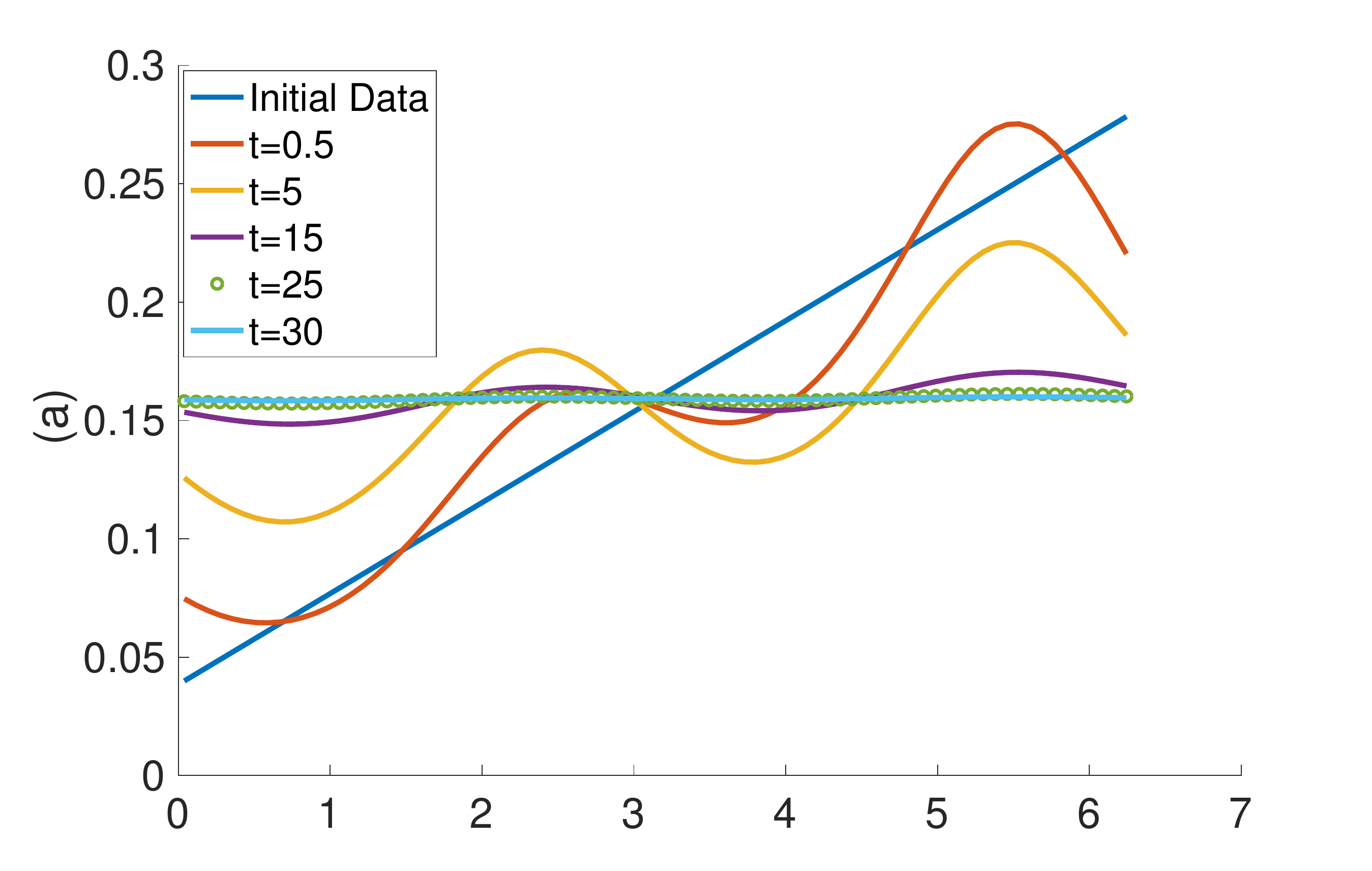}}
 \subfigure{\includegraphics[width=0.48\linewidth]{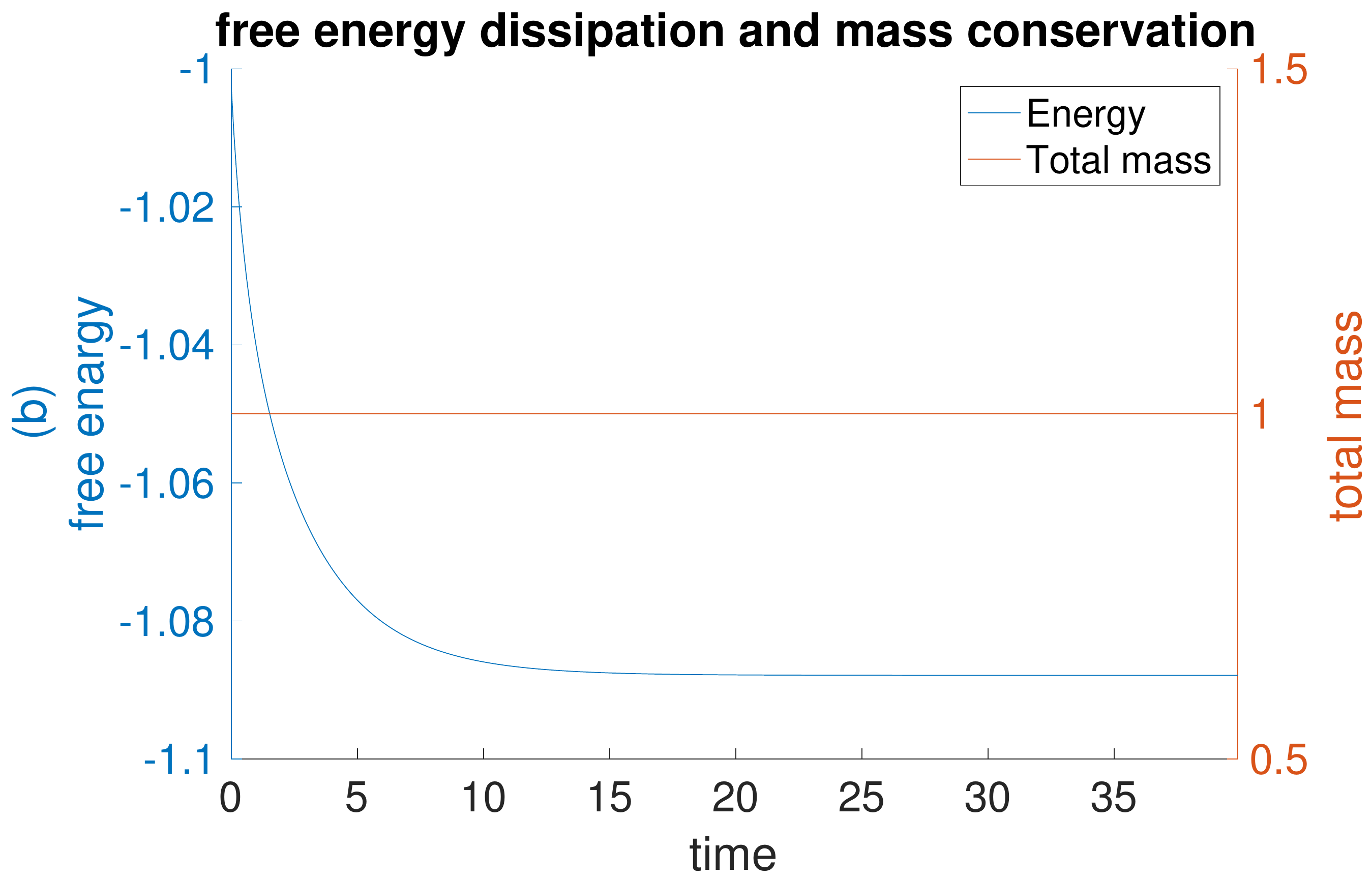}}
\end{figure}

 \begin{figure}[!tbp]
\caption{Solution evolution and energy dissipation for Example (\ref{ex53}) with $\alpha=5.$ }
  \centering
       \subfigure{\includegraphics[width=0.5\linewidth]{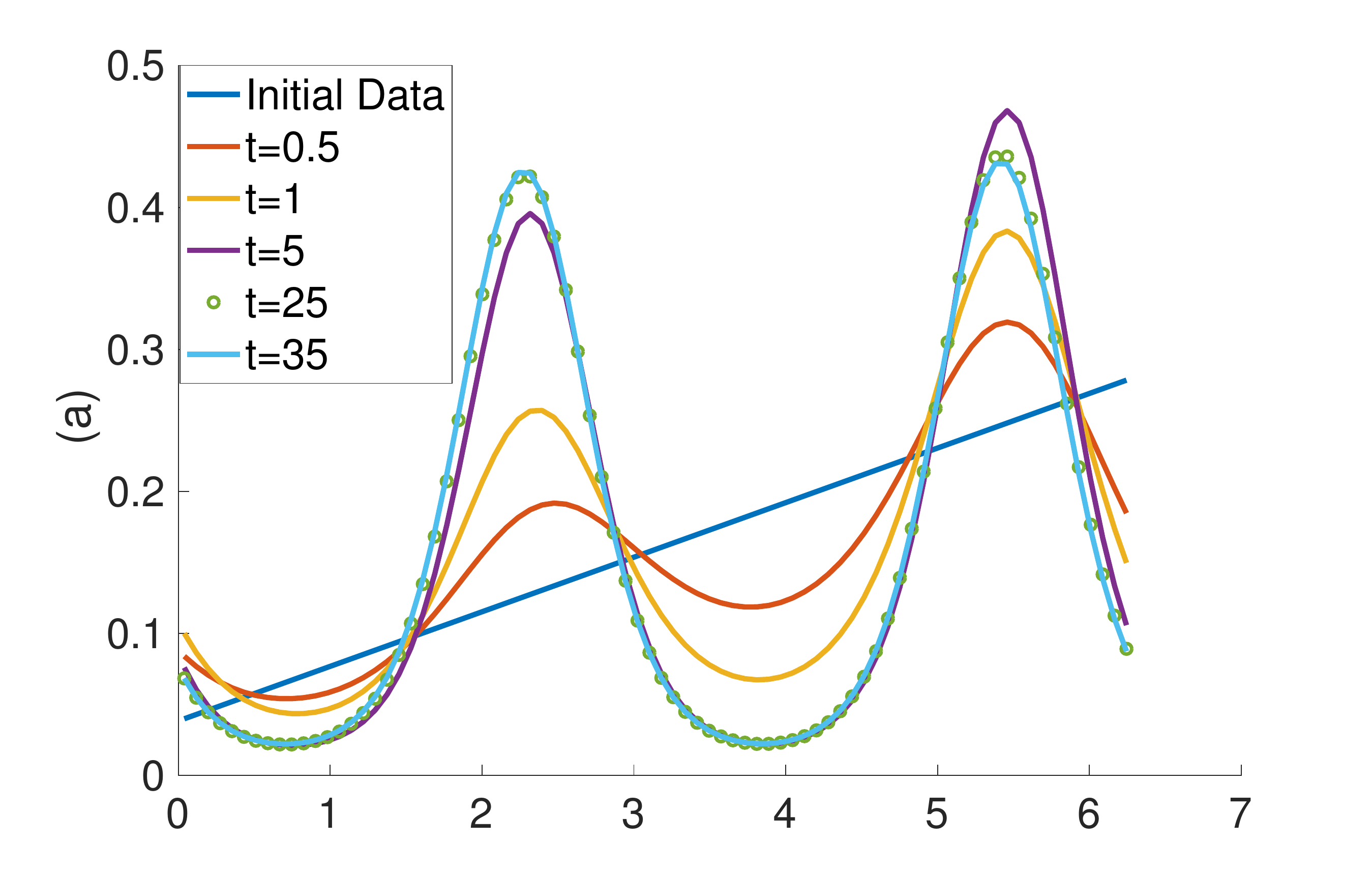}}
 \subfigure{\includegraphics[width=0.48\linewidth]{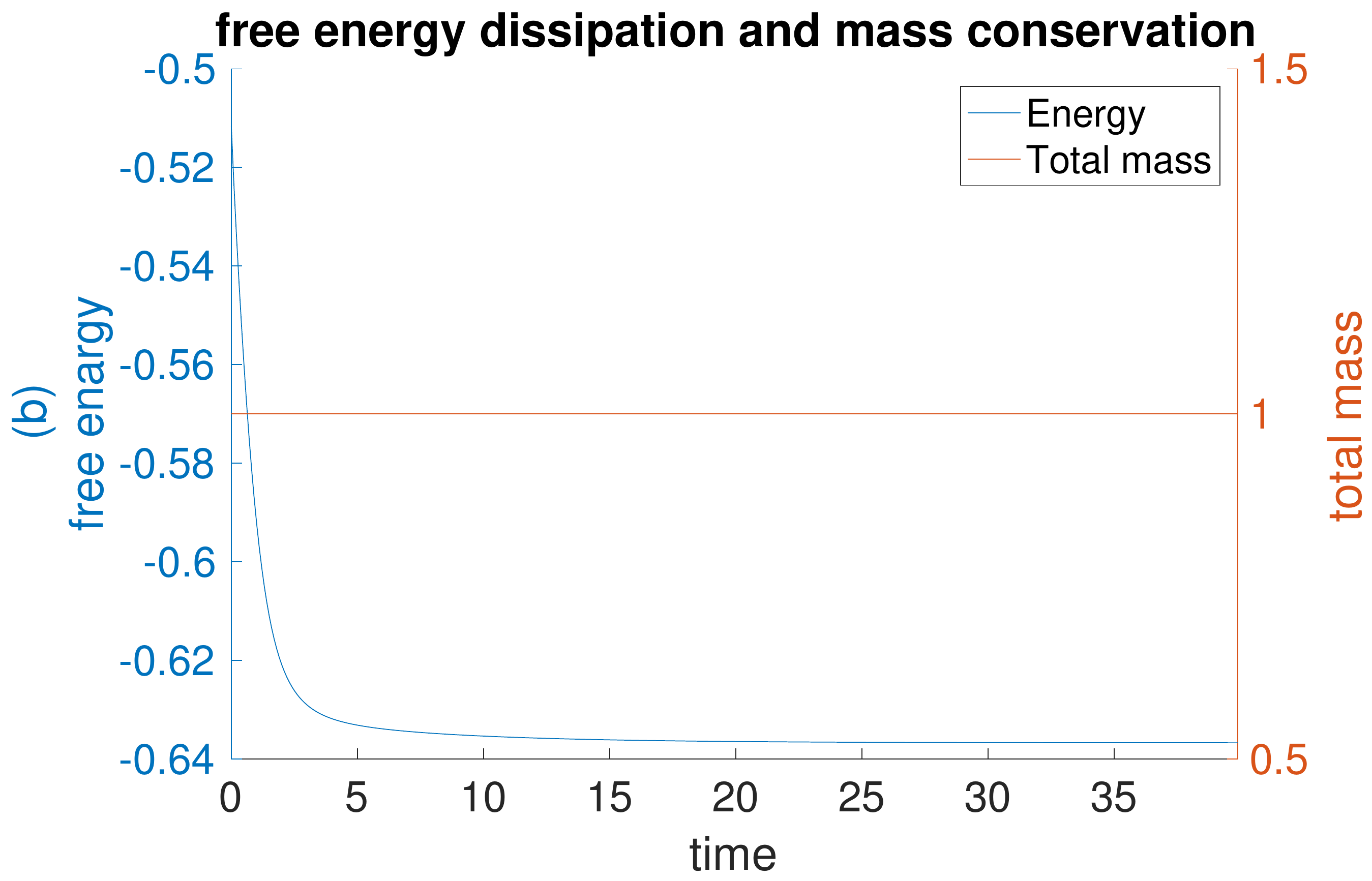}}
\end{figure}

\end{example}

\subsection{Two-dimensional tests}

\begin{example}\label{ex44} (Accuracy test)
We consider the initial value problem with source term,
\begin{equation}
\left \{
\begin{array}{rl}
  \hfill \partial_t \rho =&\nabla\cdot(\nabla\rho+\rho \nabla V(x,y))+F(x,y,t),  \hfill \ \ \   t>0, \ (x,\ y)\in [-\frac{\pi}{2},\ \frac{\pi}{2}]\times [-\frac{\pi}{2},\ \frac{\pi}{2}],\\
  \hfill \rho(x,y,0)=&2+\sin(x)\sin(y)   ,  \hfill \ \ \ \ (x,\ y)\in [-\frac{\pi}{2},\ \frac{\pi}{2}]\times [-\frac{\pi}{2},\ \frac{\pi}{2}],
 \end{array}
\right.
\end{equation}
subject to zero flux boundary conditions, here $V(x,y)=\sin(x)\sin(y),$ and
$$
F(x,y,t)=e^{-t}(2\sin^2(x)\sin^2(y)+5\sin(x)\sin(y)-\cos^2(x)\sin^2(y)-\sin^2(x)\cos^2(y)-2).
$$ 
This problem has the exact solution 
$$\rho(x,t)=e^{-t}(2+\sin(x)\sin(y)).
$$ 
We choose $\tau=0.1h^2$ in scheme (\ref{fully2}) and  $\tau=0.1h$ in scheme (\ref{PC2}). Errors and orders at $t=1$ are listed in Table 3, in this test uniform meshes with $h=h^x=h^y=\pi/N$ have been used.
 
 \begin{table}[ht]
        \centering
                \caption{Accuracy of scheme (\ref{fully2}) and (\ref{PC2}).}
\begin{tabular}{|l|l|l|l|l|l|l|l|l|}
\hline
 & \multicolumn{4}{l|}{scheme (\ref{fully2}) with $\tau=0.1h^2$} & \multicolumn{4}{l|}{scheme (\ref{PC2}) with $\tau=0.1h$} \\ \hline
 $N \times N$    &     $l^1$ error & order& $l^{\infty}$ error &order    &    $l^1$ error & order& $l^{\infty}$ error &order     \\ \hline
 $10 \times 10 $ &     0.927816E-1&-&0.175767E-1&-                      &     0.31090E-01&-&0.84728E-02&-    \\ 
$20 \times 20$&0.232384E-1&1.997&0.446660E-2&1.976           &    0.77577E-02&2.003&0.22012E-02&1.945    \\ 
 $40 \times 40$ &0.581196E-2&1.999&0.112137E-2&1.994        &    0.19368E-02&2.002&0.55550E-03&1.986   \\ 
 $80 \times 80$ &0.145297E-2&2.000&0.280607E-3&1.999         &     0.48558E-03&1.996&0.13975E-03&1.991   \\ \hline
\end{tabular}
\end{table}

\end{example}

Finally we mention that there is a class of equations in which the interaction is modeled through a potential governed by the Poisson equation. The celebrated model is the Patlak-Keller-Segel system of the chemotaxis \cite{Horst1,Horst2}. The original model is a coupled parabolic system, and 
the one related to our model equation (\ref{mainmodel}) is the parabolic-elliptic version of the form (see e.g., \cite{PKS3}) 
\begin{equation}\label{PKS11}
\left \{
\begin{array}{rl}
\hfill \partial_t \rho =&\Delta \rho-\nabla\cdot(\chi\rho \nabla c),  \hfill \ \ \   t>0, \ \mathbf{x}\in  \R^2,\\
\hfill -\Delta c=& \rho,  \hfill \ \ \  \\
\hfill \rho(\mathbf{x},0)=&\rho_0(\mathbf{x})   ,  \hfill \ \ \ \ \mathbf{x}\in  \R^2.
\end{array}
\right.
\end{equation}
Here, $\rho(\mathbf{x},t)$ is the cell density, $c(\mathbf{x},t)$ is the chemical attractant concentration, the parameter $\chi>0$ is the sensitivity of bacteria to the chemical attractant.  It has been shown in \cite{PKS} that the solution behavior of problem (\ref{PKS11}) is quite different when crossing a critical mass.  
If the initial mass $M=\int_{\mathbb{R}^2} \rho_0(x, y)dxdy$  is smaller than a critical value $M_c =8\pi/\chi$, then the solution exists globally. When $M > M_c$, the solution will blow up in finite time, which is referred to as chemotactic collapse.

\begin{example}\label{ex45} {(Patlak$-$Keller$-$Segal system)}. In this example, we test the method's capacity in capturing solution concentrations for the Patlak$-$Keller$-$Segal system (\ref{PKS11}).   Using the Green function for the Poisson equation, this system can be reformulated as (\ref{mainmodel}) with $V=0$ and 
\begin{equation}\label{Green}
W(x,y)=\frac{\chi}{2\pi}\log(\sqrt{x^2+y^2}).
\end{equation}
  In our simulation, we restrict to a bounded domain $\Omega$ subject to zero flux boundary conditions, using formulation (\ref{mainmodel21}) with $ V(x,y)=0$ and $W$ defined in (\ref{Green}). We fix $\chi=1$ and consider both the sub-critical case with 
\begin{align*}
\rho_0(x,y)=\left \{
\begin{array}{rl}
 &2(\pi-0.2),  \hfill \ \ \   \ (x,y) \in [-1,1]\times[-1,1], \\\\
& 0,  \hfill \ \ \   \  (x,y) \in \Omega \backslash [-1,1]\times[-1,1],
\end{array}
\right.
\end{align*}
on $\Omega =[-5,5]\times[-5,5]$,  and super-critical case with 
\begin{align*}
\rho_0(x,y)=\left \{
\begin{array}{rl}
 &2(\pi+0.2),  \hfill \ \ \   \ (x,y) \in [-1,1]\times[-1,1], \\\\
& 0,  \hfill \ \ \   \  (x,y) \in \Omega \backslash [-1,1]\times[-1,1],
\end{array}
\right.
\end{align*}
on $\Omega =[-1.5,1.5]\times[-1.5,1.5]$, for which we know that the solution blows-up at finite time.

We take time step $\tau= 0.01$, and set $N_x=N_y=51$ so that a single cell is located at the center of the computational domain, where one can view  a clear picture of the blow-up phenomena in super-critical case. In Fig.5 are snap shots of numerical solutions in the sub-critical case at $t=0,\ 2, \ 8, \ 12, \ 16$, from which we observe that the numerical solution dissipates in time, the last picture in Fig.5 shows mass conservation and energy dissipation. In Fig.6 are snap shots of numerical solutions in super-critical case at $t=0,\ 0.5,\ 1, \ 1.5, \ 2$, we observe that numerical solutions tend to concentrate at the origin.   

Let us remark that in \cite{DG} the same concentration phenomena was observed, using a DG method for this problem with periodic boundary conditions. Different boundary conditions do not affect the concentration profile since the solution is compactly supported in our setting.  In the super-critical case, 
the peak in our result is slightly lower than that captured in \cite{DG}, this is expected because the solution is concentrated at a single point,  and cell averaging near the origin can decrease the height of the peak.

\begin{figure}[!tbp]
\caption{Solution evolution  for Example \ref{ex45} (sub-critical).}
\centering  
         \subfigure{\includegraphics[width=0.45\linewidth]{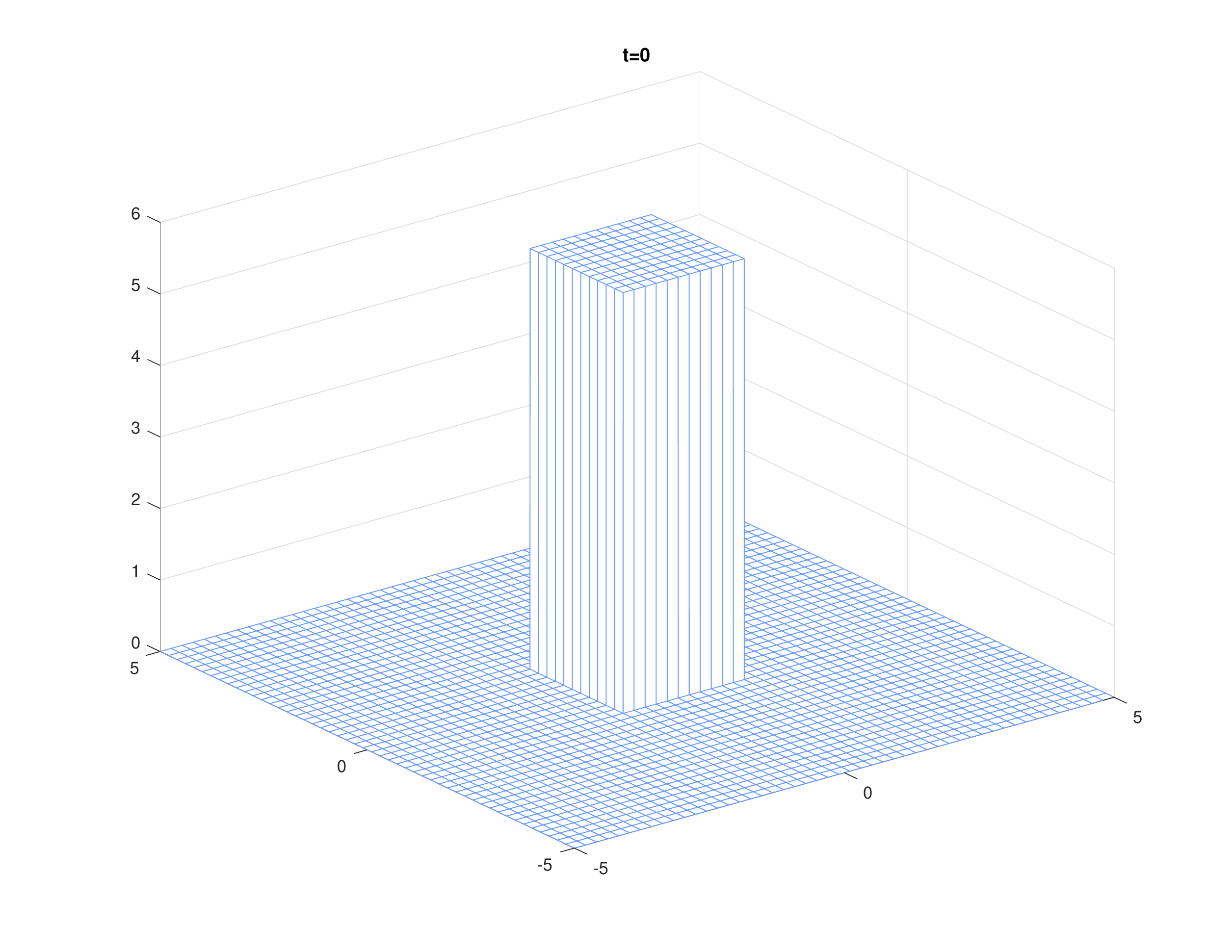}}
      \subfigure{\includegraphics[width=0.45\linewidth]{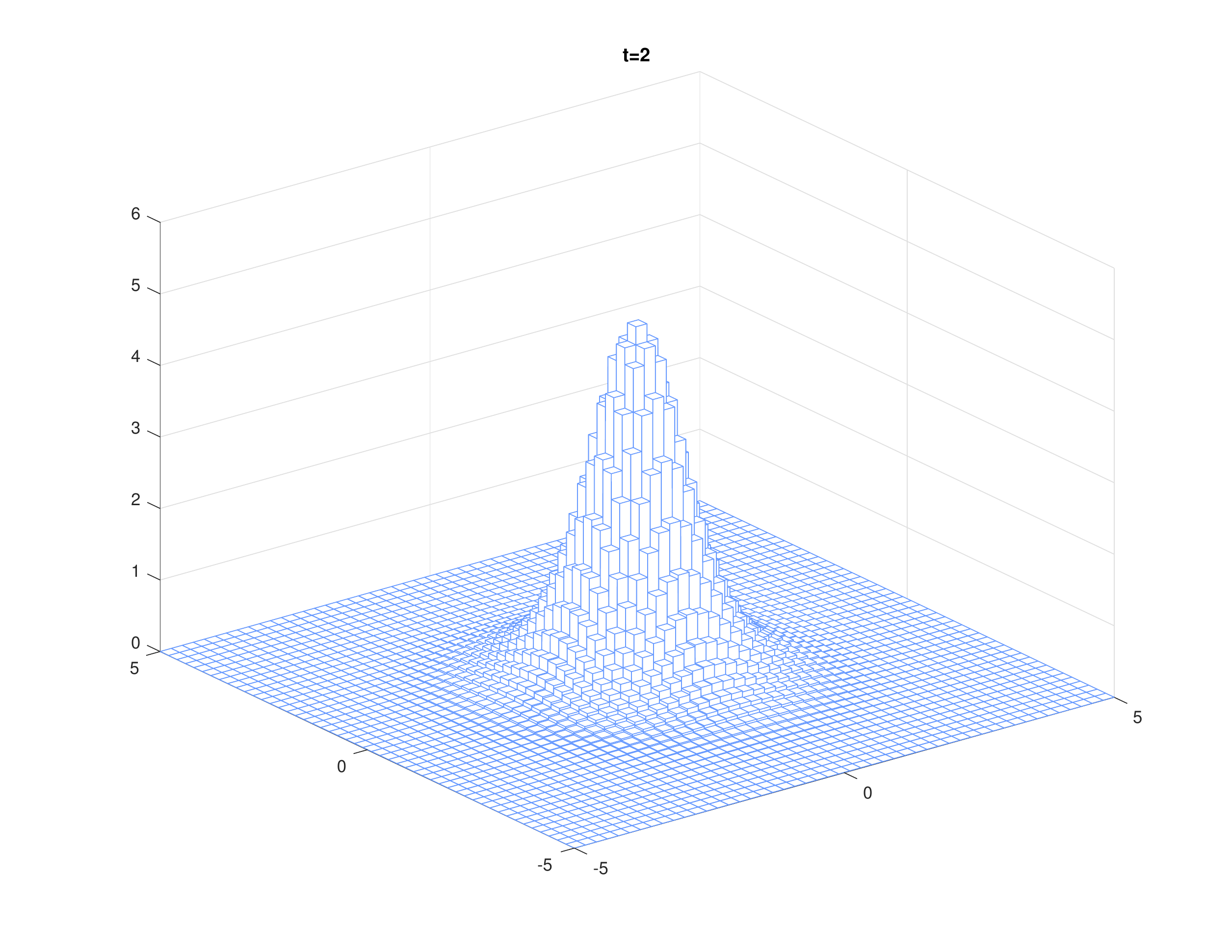}}\\
      \subfigure{\includegraphics[width=0.45\linewidth]{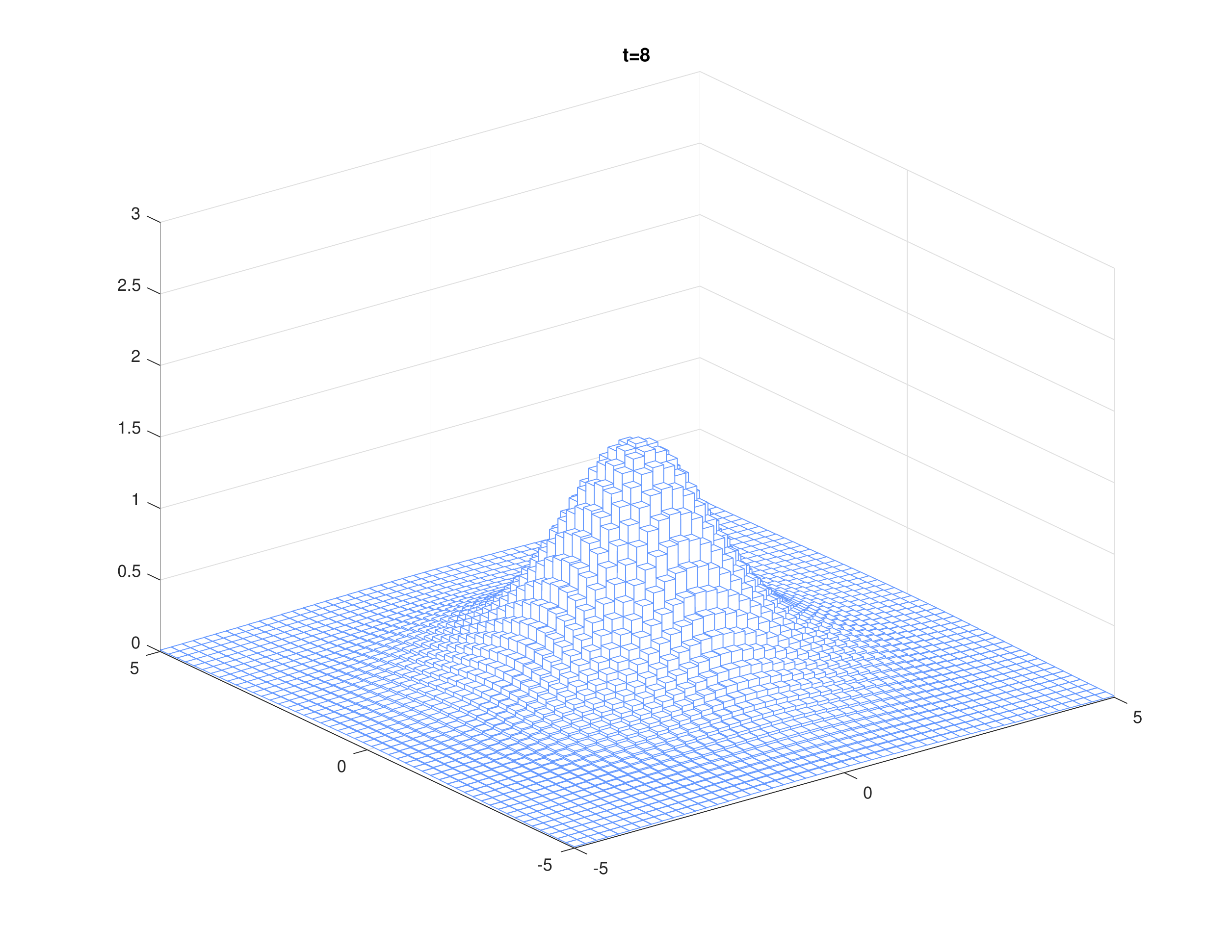}}
               \subfigure{\includegraphics[width=0.45\linewidth]{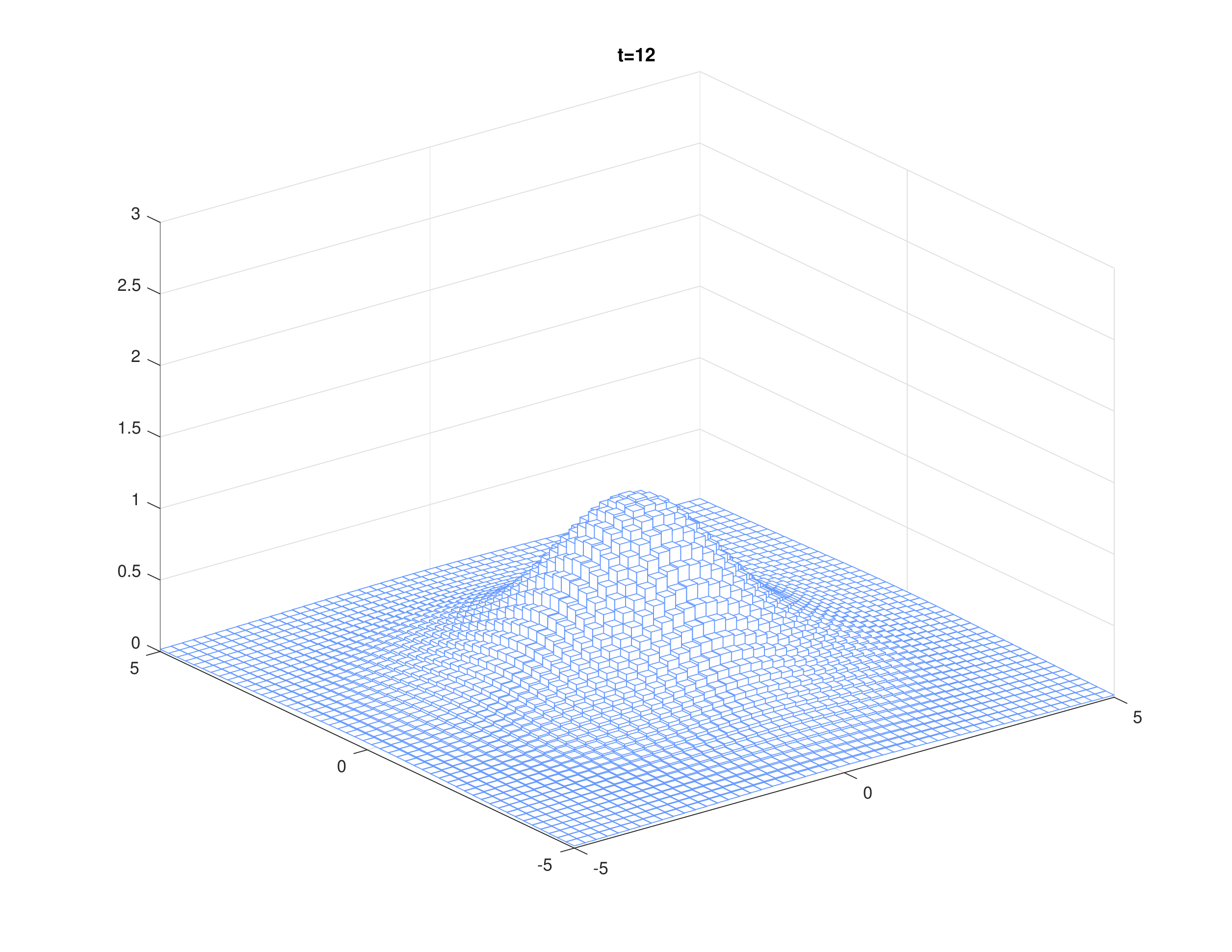}}\\
      \subfigure{\includegraphics[width=0.45\linewidth]{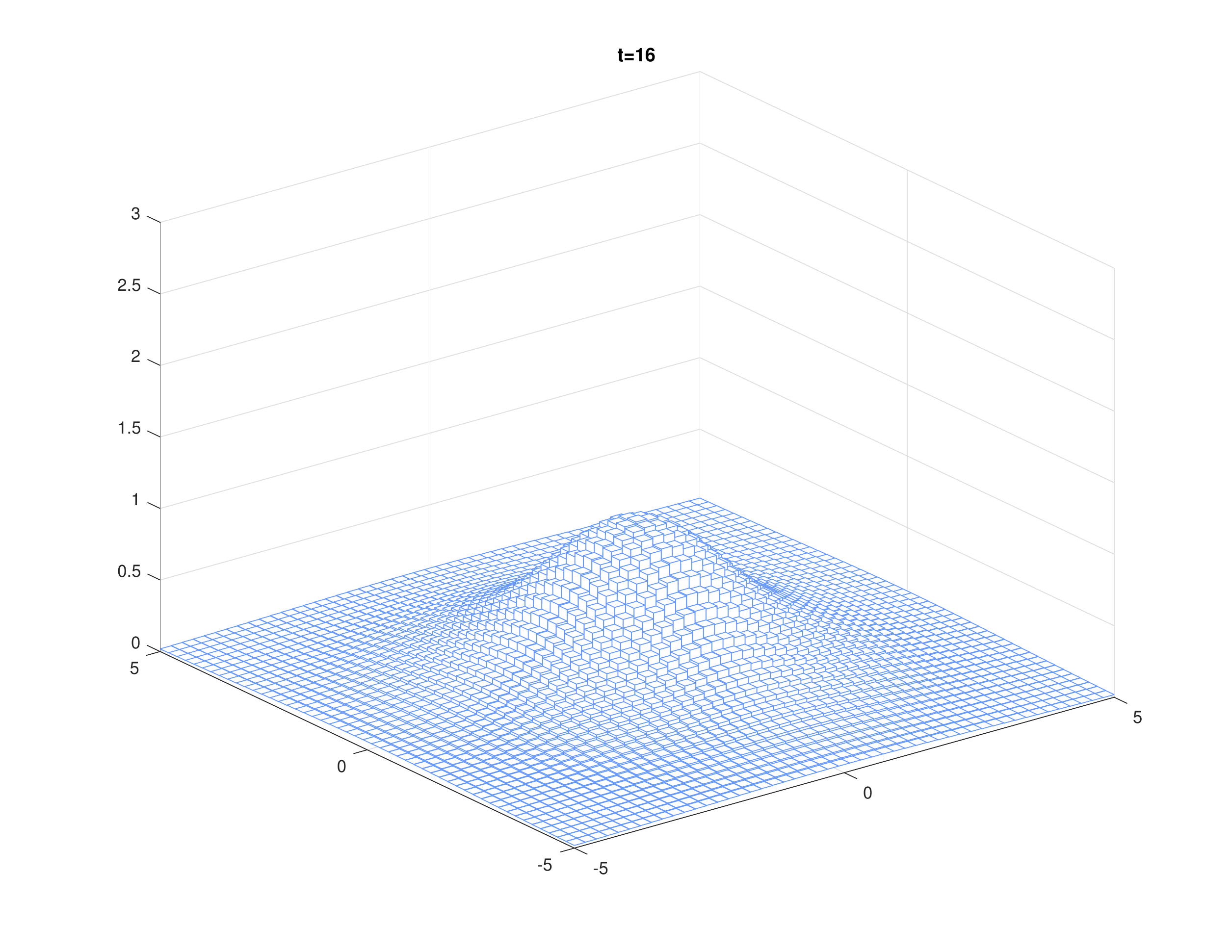}}
      \subfigure{\includegraphics[width=0.4\linewidth]{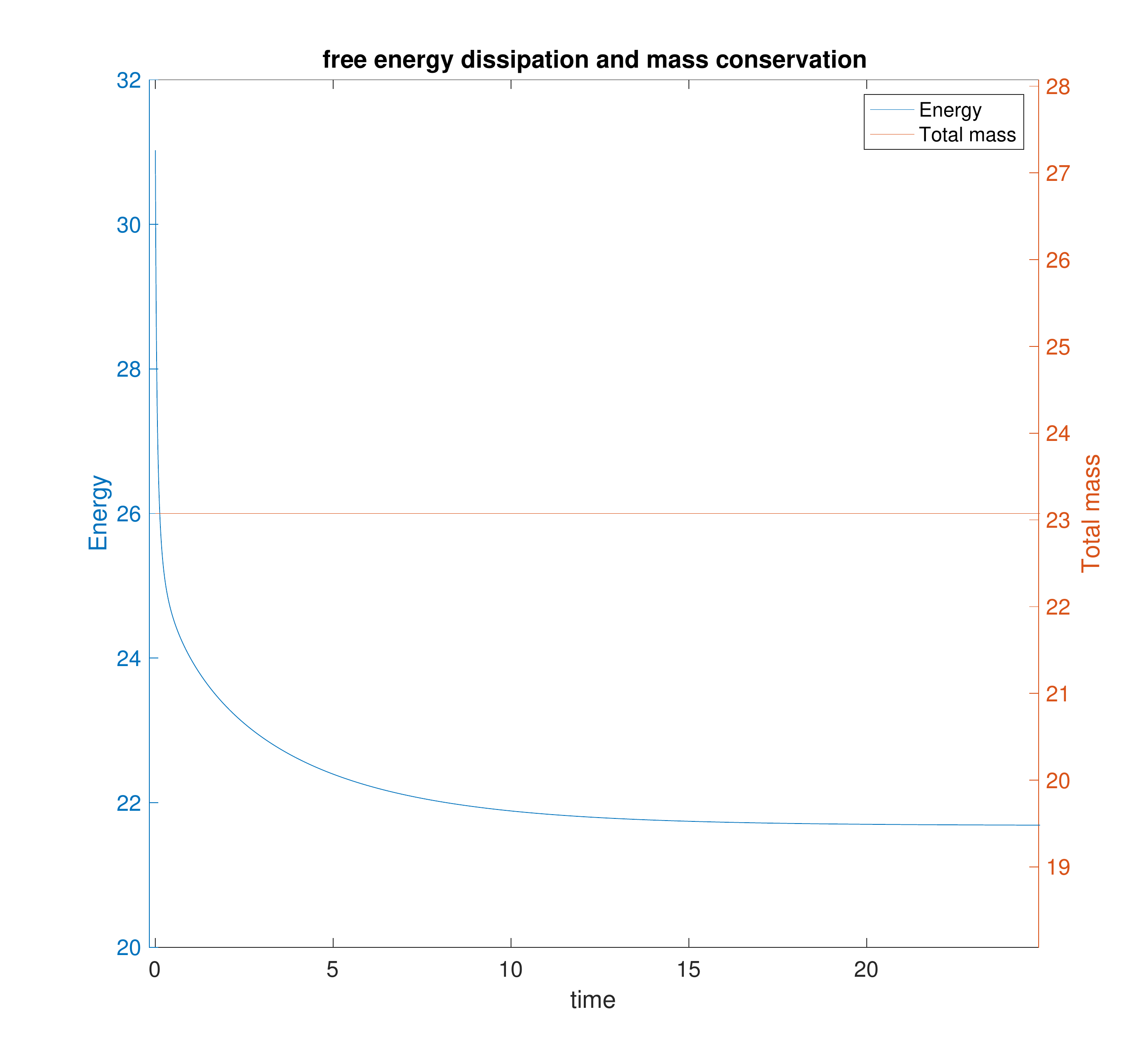}}
\end{figure}

\begin{figure}[!tbp]
\caption{Solution evolution  for Example \ref{ex45} (super-critical).}
\centering  
         \subfigure{\includegraphics[width=0.45\linewidth]{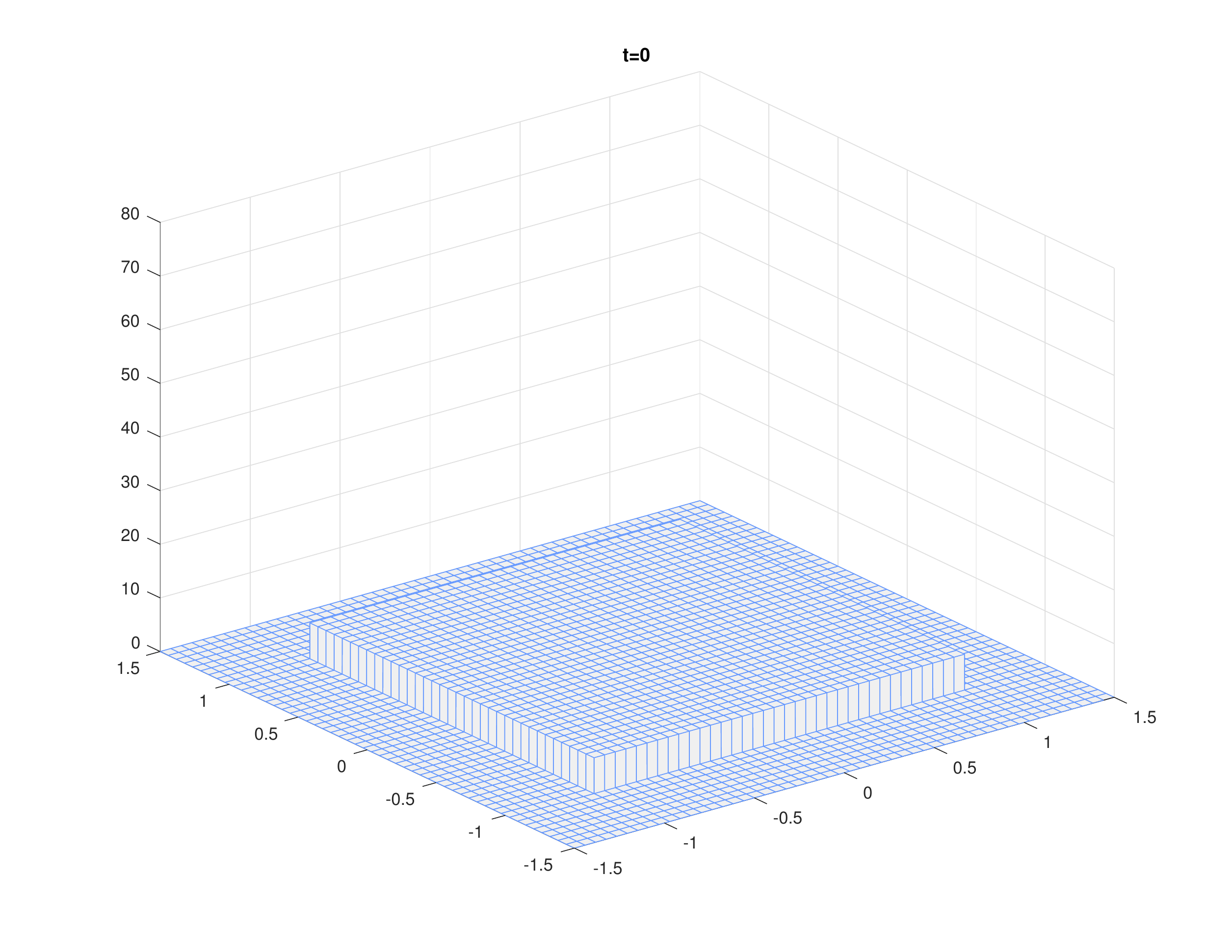}}
      \subfigure{\includegraphics[width=0.45\linewidth]{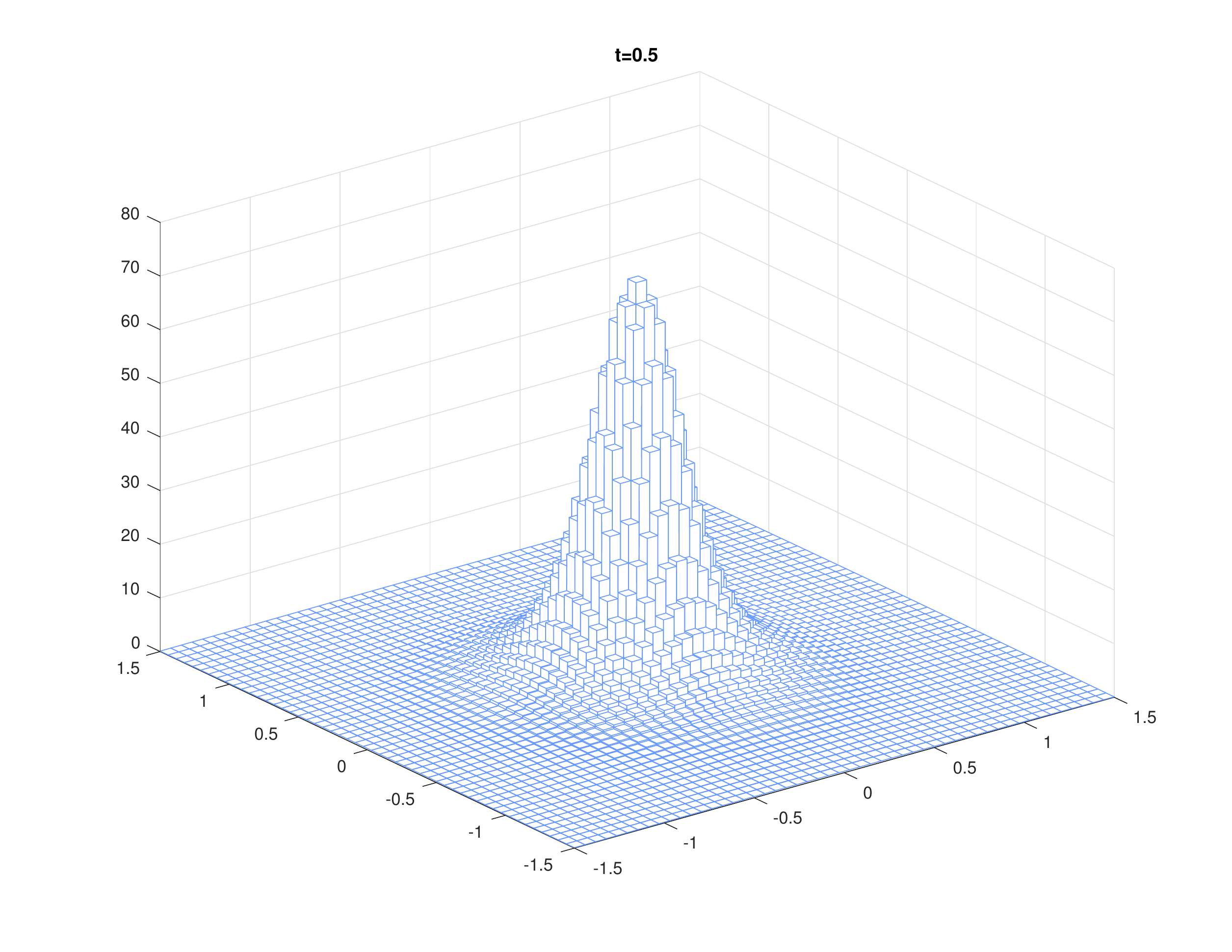}}\\
      \subfigure{\includegraphics[width=0.45\linewidth]{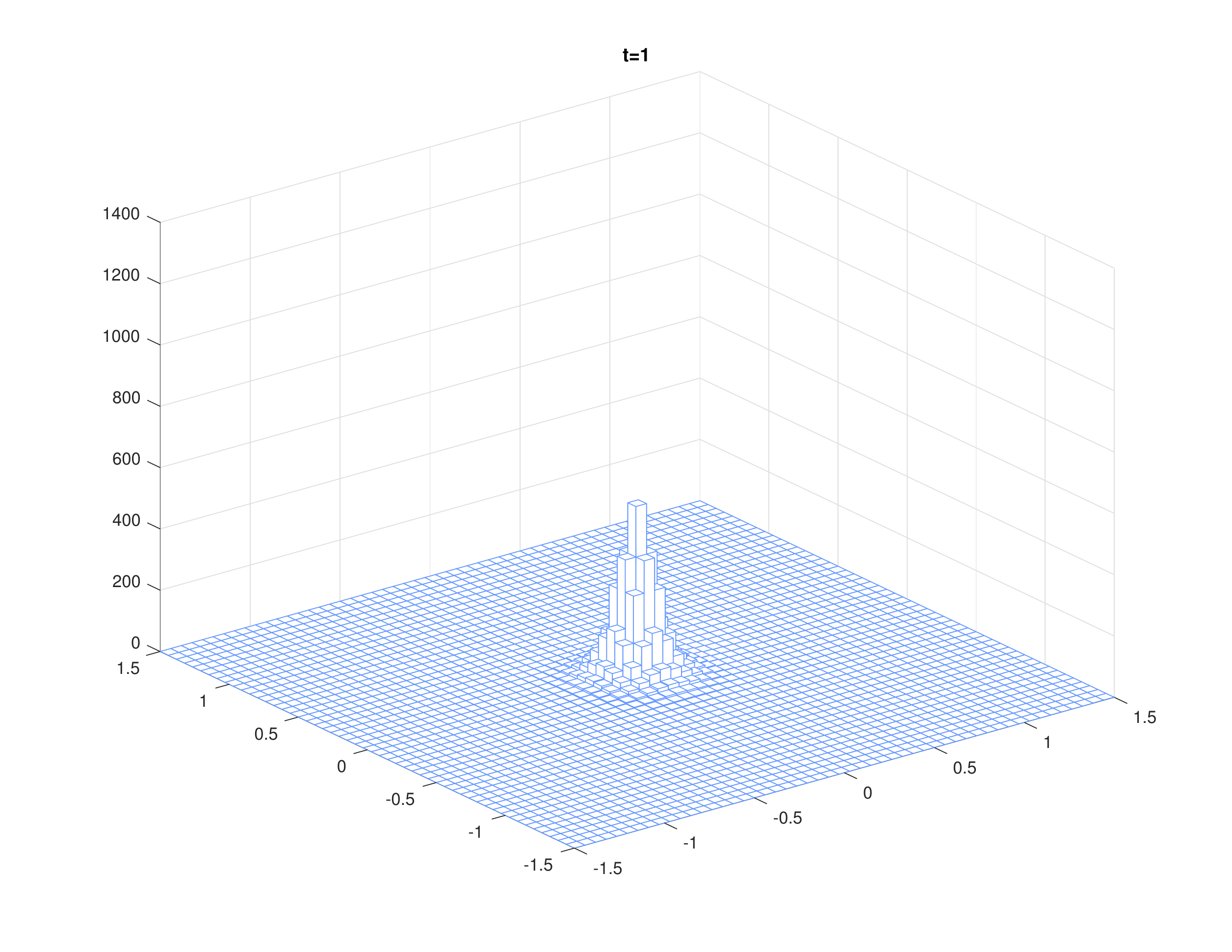}}
               \subfigure{\includegraphics[width=0.45\linewidth]{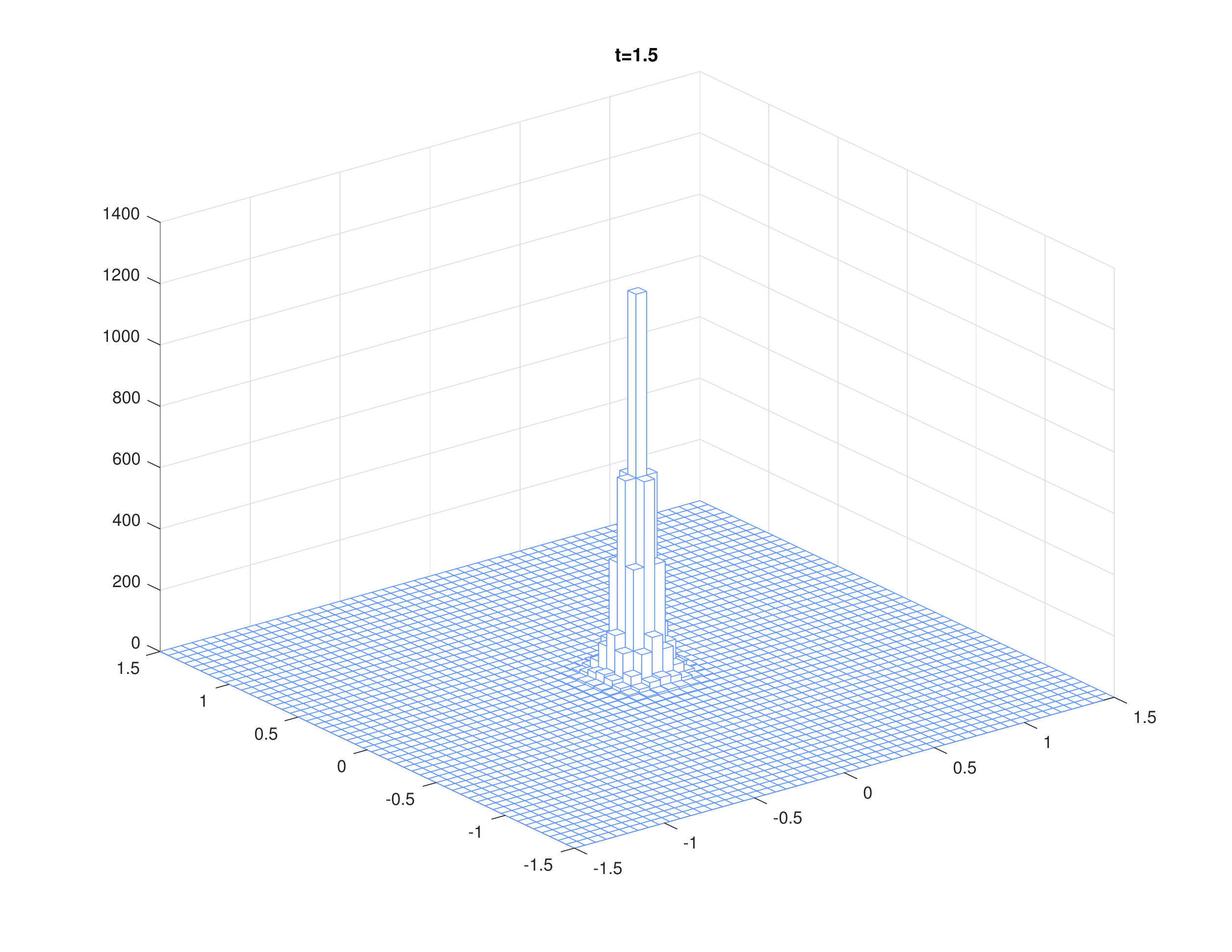}}\\
      \subfigure{\includegraphics[width=0.45\linewidth]{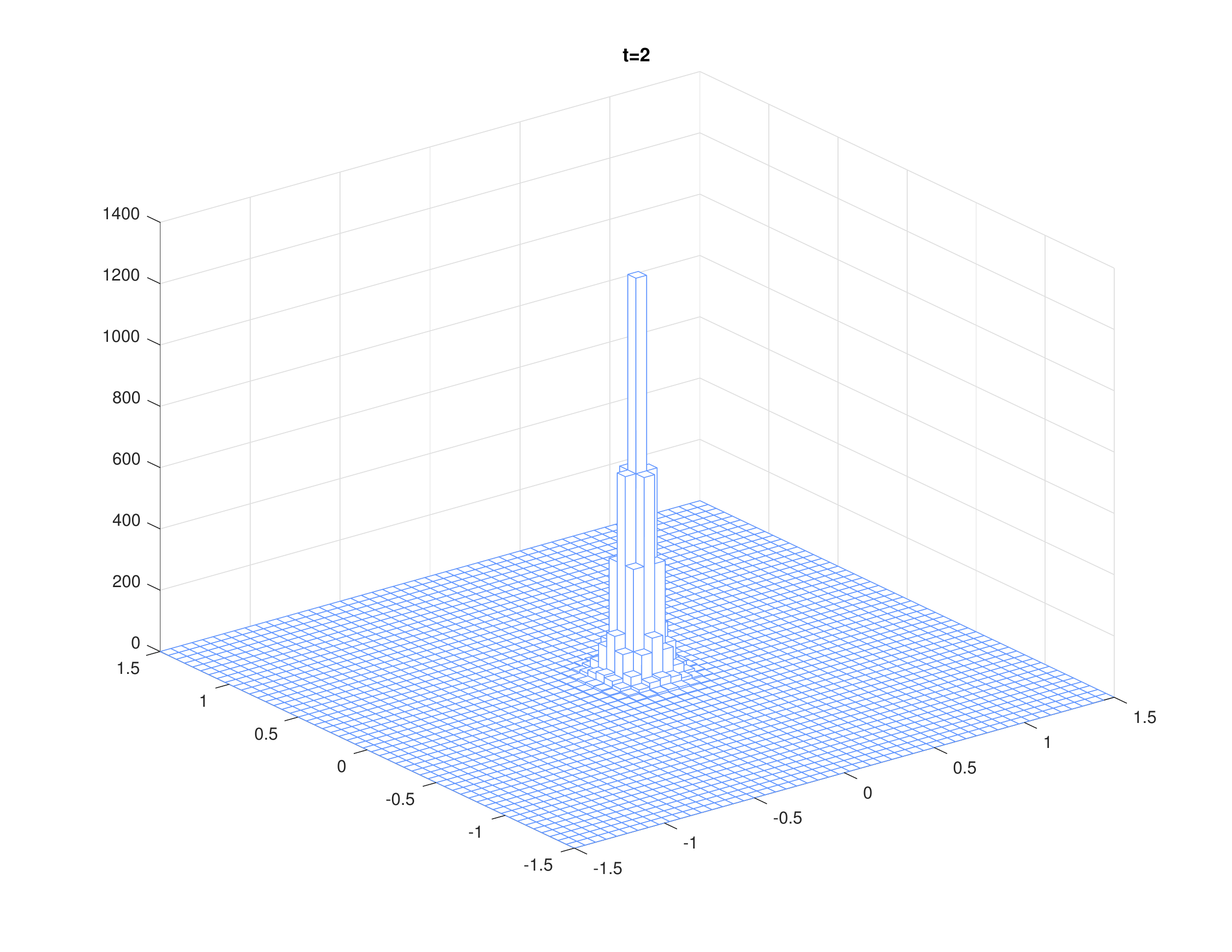}}
      \subfigure{\includegraphics[width=0.4\linewidth]{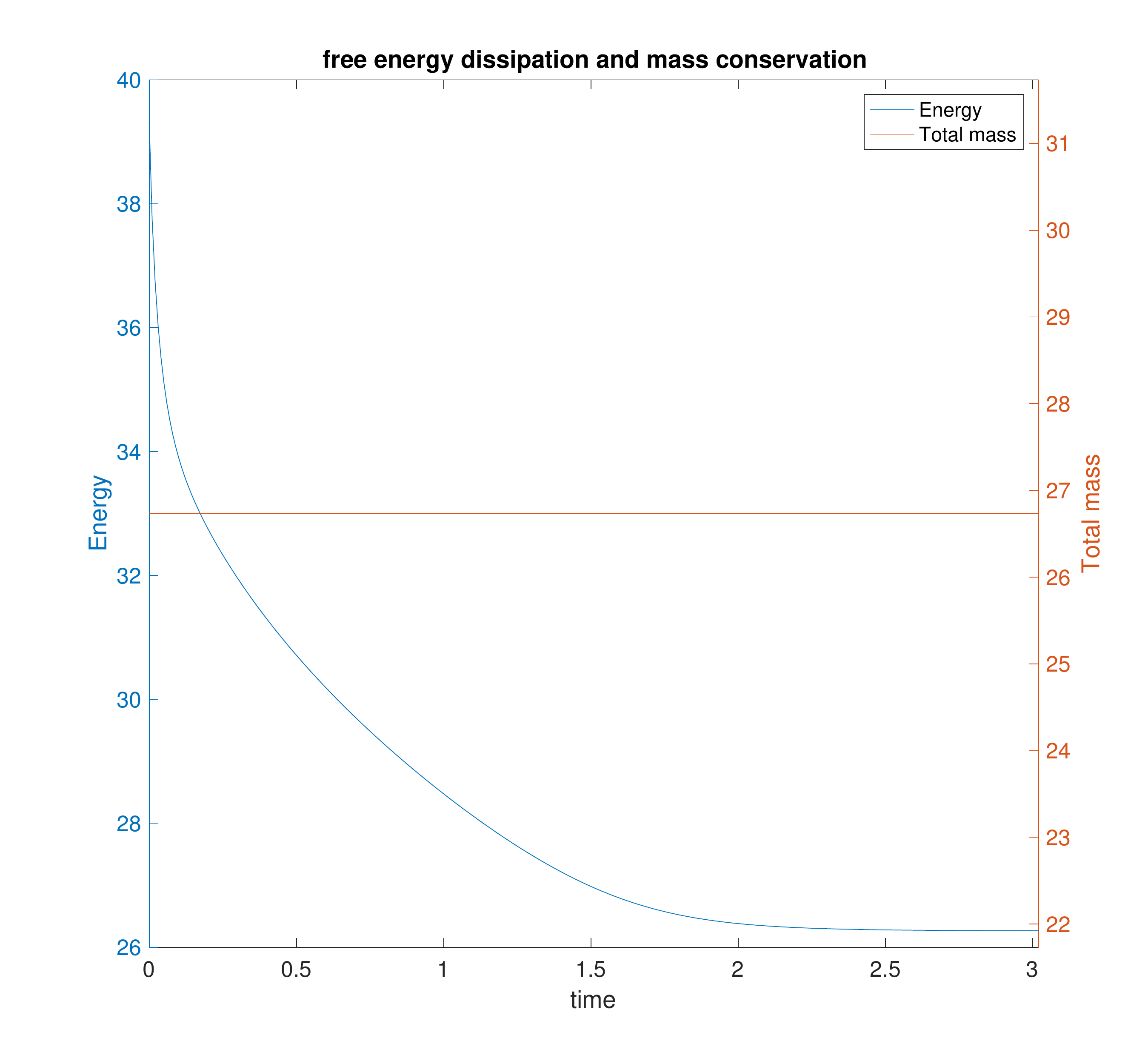}}
\end{figure}

\end{example}

\section{Concluding remarks}
In this paper, we have developed positive and free energy satisfying schemes for diffusion equations with interaction potentials; since such equations are governed by a free energy dissipation law and are featured with non-negative solutions.  Based on the non-logarithmic Landau reformulation of the model, we constructed a simple, easy-to-implement fully discrete numerical scheme (first order in time) which proved to satisfy all three desired properties of the continuous model: mass conservation, free energy dissipation and non-negativity, without a strict time step restriction.  For a fully second order  (in both time and space) scheme ,we used a local scaling limiter to restore solution positivity when necessary.  Moreover, we rigorously proved that the limiter does not destroy the second order accuracy.  Numerical examples have demonstrated the superior performance of these schemes, in particular, the three solution properties numerically confirmed are consistent with our theoretical findings.

\section*{Acknowledgments}
 This research was supported by the National Science Foundation under Grant DMS1312636.
\bigskip
\bibliographystyle{abbrv}

\end{document}